\documentclass[10pt, oneside, reqno]{article}

\usepackage{mathrsfs}

\usepackage{amsfonts,amsmath,amstext,amsbsy,amsthm,amscd,amssymb}
\usepackage{color}
\usepackage{pdfsync}

\definecolor{darkblue}{rgb}{0.0,0.0,0.7}

\RequirePackage[%
colorlinks = true,%
linkcolor = darkblue,%
citecolor = darkblue,%
urlcolor = darkblue, %
]{hyperref}%

\hypersetup{%
  pdfauthor = {Sylvain Delattre, St\'ephane Ga\"iffas},%
  pdfcreator = {pdflatex},%
  pdfproducer = {pdflatex}}

\let\epsilon=\varepsilon
\let\ep=\epsilon
\let\phi=\varphi
\let\al=\alpha  \let\Ga=\Gamma 
\let\de=\delta
  \let\ga=\gamma
\let\Om=\Omega \let\om=\omega 
\let\si=\sigma  \let\ze=\zeta
\let\la=\lambda  
\let\tilde=\widetilde

\newcommand{\abs}[1]{\lvert#1\rvert}
\newcommand{\bigabs}[1]{\bigl\lvert#1\bigr\rvert}
\newcommand{\Bigabs}[1]{\Bigl\lvert#1\Bigr\rvert}

\newcommand{\bigpar}[1]{\bigl(#1\bigr)}
\newcommand{\Bigpar}[1]{\Bigl(#1\Bigr)}

\newcommand{\bigcro}[1]{\bigl[#1\bigr]}
\newcommand{\Bigcro}[1]{\Bigl[#1\Bigr]}

\newcommand{\norm}[1]{\|#1\|}

\newcommand{\field}[1]{\mathbb{#1}}
\newcommand{\R}{\field{R}}
\newcommand{\N}{\field{N}}

\newcommand{\G}{{\mathcal G}}

\newcommand{\F}{{\mathscr{F}}}

\newcommand{\prodsca}[1]{\langle #1 \rangle}
\newcommand{\HH}{{\mathcal H}}
\newcommand{\Hb}{H}

\def\der^#1_#2{\frac{\partial^{#1}}{\partial {#2}^{#1}}}
\newcommand{\ind}[1]{\mathbf 1_{#1}}
\newcommand{\set}[1]{\left\{#1\right\}}

\newcommand{\var}{{\rm Var}\,}

\theoremstyle{plain}
\newtheorem{theorem}{Theorem} \newtheorem{corollary}{Corollary}
\newtheorem{lemma}{Lemma}
\newtheorem{prop}{Proposition}
\newtheorem{assumption}{Assumption}{\bf}{\rm}%
\theoremstyle{definition}
\theoremstyle{remark}
\newtheorem{remark}{Remark}

\newtheorem{example}{Example}
{\smallskip}%
\newcommand{\E}{\field{E}}
\renewcommand{\P}{\field{P}}

\begin{document}

\bibliographystyle{plain}

\title{Nonparametric regression with martingale increment errors}

\author{Sylvain Delattre$ ^1$, St\'ephane Ga\"iffas$ ^2$}

\footnotetext[1]{Universit\'e Paris Diderot - Paris 7, Laboratoire de
  Probabilit\'es et Mod\`eles Al\'eatoires. \emph{email}:
  \texttt{delattre@math.jussieu.fr}}

\footnotetext[2]{Universit\'e Pierre et Marie Curie - Paris~6,
  Laboratoire de Statistique Th\'eorique et Appliqu\'ee. \emph{email}:
  \texttt{stephane.gaiffas@upmc.fr}}

\footnotetext[3]{This work is supported in part by French Agence
  Nationale de la Recherche (ANR) ANR Grant \textsc{``Prognostic''}
  ANR-09-JCJC-0101-01. (\texttt{http://www.lsta.upmc.fr/prognostic/index.php})
}

\maketitle

\begin{abstract}
  We consider the problem of adaptive estimation of the regression
  function in a framework where we replace ergodicity assumptions
  (such as independence or mixing) by another structural assumption on
  the model. Namely, we propose adaptive upper bounds for kernel
  estimators with data-driven bandwidth (Lepski's selection rule) in a
  regression model where the noise is an increment of martingale. It
  includes, as very particular cases, the usual i.i.d. regression and
  auto-regressive models. The cornerstone tool for this study is a new
  result for self-normalized martingales, called ``stability'', which
  is of independent interest. In a first part, we only use the
  martingale increment structure of the noise. We give an adaptive
  upper bound using a random rate, that involves the occupation time
  near the estimation point. Thanks to this approach, the theoretical
  study of the statistical procedure is disconnected from usual
  ergodicity properties like mixing. Then, in a second part, we make a
  link with the usual minimax theory of deterministic rates. Under a
  $\beta$-mixing assumption on the covariates process, we prove that
  the random rate considered in the first part is equivalent, with
  large probability, to a deterministic rate which is the usual
  minimax adaptive one.
  \\

  \noindent%
\end{abstract}

\section{Introduction}

\subsection{Motivations}

In the theoretical study of statistical or learning algorithms,
stationarity, ergodicity and concentration inequalities are
assumptions and tools of first importance. When one wants to obtain
asymptotic results for some procedure, stationarity and ergodicity of
the random process generating the data is mandatory. Using extra
assumptions, like moments and boundedness conditions, concentration
inequalities can be used to obtain finite sample results. Such tools
are standard when the random process is assumed to be i.i.d., like
Bernstein's or Talagrand's inequality (see \cite{ledoux_talagrand91},
\cite{van_der_vaart_wellner96} and \cite{MR1419006}, among others). To
go beyond independence, one can use a mixing assumption in order to
``get back'' independence using coupling, see \cite{MR1312160}, so
that, roughly, the ``independent data tools'' can be used again. This
approach is widely used in nonparametric statistics, statistical
learning theory and time series analysis.

The aim of this paper is to replace stationarity and ergodicity
assumptions (such as independence or mixing) by another structural
assumption on the model. Namely, we consider a regression model where
the noise is an increment of martingale. It includes, as very
particular cases, the usual i.i.d. regression and the auto-regressive
models. The cornerstone tool for this study is a new result, called
``stability'', for self-normalized martingales, which is of
independent interest. In this framework, we study kernel estimators
with a data-driven bandwidth, following the Lepski's selection rule,
see~\cite{lepski90}, \cite{lepski_mammen_spok_97}.

The Lepski's method is a statistical algorithm for the construction of
optimal adaptive estimators. It was introduced in
\cite{lepski88,lepski90,lepski_92}, and it provides a way to select
the bandwidth of a kernel estimator from the data. It shares the same
kind of adaptation properties to the inhomogeneous smoothness of a
signal as wavelet thresholding rules, see
\cite{lepski_mammen_spok_97}. It can be used to construct an adaptive
estimator of a multivariate anisotropic signal, see
\cite{kerk_lepski_picard01}, and recent developments shows that it can
be used in more complex settings, like adaptation to the
semi-parametric structure of the signal for dimension reduction, or
the estimation of composite functions, see \cite{MR2449122},
\cite{composite}. In summary, it is commonly admitted that Lepski's
idea for the selection of a smoothing parameter works for many
problems. However, theoretical results for this procedure are mostly
stated in the idealized model of Gaussian white noise, excepted for
\cite{MR2339297}, where the model of regression with a random design
was considered. As far as we know, nothing is known on this procedure
in other settings: think for instance of the auto-regressive model or
models with dependent data.

Our approach is in two parts: in a first part, we consider the problem
of estimation of the regression function. We give an adaptive upper
bound using a random rate, that involves the occupation time at the
estimation point, see Theorem~\ref{thm:upper_bound}. In this first
part, we only use the martingale increment structure of the noise, and
not stationarity or ergodicity assumptions on the
observations. Consequently, even if the underlying random process is
transient (e.g. there are few observations at the estimation point),
the result holds, but the occupation time is typically small, so that
the random rate is large (and eventually not going to zero as the
sample size increases). The key tool is a new result of stability for
self-normalized martingales stated in Theorem~\ref{thm:key}, see
Section~\ref{sec:martingale}. It works surprisingly well for the
statistical application proposed here, but it might give new results
for other problems as well, like upper bounds for procedures based on
minimization of the empirical risk, model selection (see
\cite{massart03}), etc. In a second part (Section~\ref{sec:examples}),
we make a link with the usual minimax theory of deterministic rates.
Using a $\beta$-mixing assumption, we prove that the random rate used
in Section~\ref{sec:main-results} is equivalent, with a large
probability, to a deterministic rate which is the usual adaptive
minimax one, see Proposition~\ref{prop:det-rate}.

The message of this paper is twofold. First, we show that the kernel
estimator and Lepski's method are very robust with respect to the
statistical properties of the model: they does not require
stationarity or ergodicity assumptions, such as independence or mixing
to ``do the job of adaptation'', see
Theorem~\ref{thm:upper_bound}. The second part of the message is that,
for the theoretical assessment of an estimator, one can use
advantageously a theory involving random rates of convergence. Such a
random rate naturally depends on the occupation time at the point of
estimation (=the local amount of data), and it is ``almost
observable'' if the smoothness of the regression were to be known. An
ergodicity property, such as mixing, shall only be used in a second
step of the theory, for the derivation of the asymptotic behaviour of
this rate (see Section~\ref{sec:examples}). Of course, the idea of
random rates for the assessment of an estimator is not new. It has
already been considered in \cite{MR1802549, RePEc:wpa:wuwpem:0411007}
for discrete time and in \cite{dhk02} for diffusion models. However,
this work contains, as far as we know, the first result concerning
adaptive estimation of the regression with a martingale increment
noise.

\subsection{The model}
\label{sec:model}

Consider sequences $(X_k)_{k\ge 0}$ and $(Y_k)_{k\ge 1}$ of random
variables respectively in $\mathbb R^d$ and $\mathbb R$, both adapted
to a filtration $(\F_k)_{k\ge 0}$, and such that for all $k\ge 1$:
\begin{equation}
  \label{regmart}
  Y_k = f(X_{k-1}) + \ep_k,
\end{equation}
where the sequence $(\varepsilon_k)_{k \geq 1}$ is a $(\mathscr
F_k)$-martingale increment:
\begin{equation*}
  \E( | \ep_k | | \F_{k-1} ) < \infty \; \text{ and } \; \E
  ( \ep_k | \F_{k-1} ) = 0,
\end{equation*}
and where $f : \R^d \rightarrow \R$ is the unknown function of
interest. We study the problem of estimation of $f$ at a point $x \in
\R^d$ based on the observation of $(Y_1,\dots, Y_N)$ and $(X_0,\dots,
X_{N-1})$, where $N \ge 1$ is a finite $(\mathscr F_k)$-stopping time.
This allows for ``sample size designing'', see
Remark~\ref{rem:stoppingtime} below. The analysis is conducted under
the following assumption on the sequence $(\ep_k)_{k \geq 1}$:
\begin{assumption}
  \label{hypmoments}
  There is a $(\mathscr F_k)$-adapted sequence $(\sigma_k)_{k \geq
    0}$, assumed to be observed, of positive random variables and
  $\mu, \gamma > 0$ such that\textup:
  \begin{equation*}
    \E\Bigcro{\exp\Bigpar{\mu \frac{\ep_k^2} {\si^2_{k-1}} } \mid
      \F_{k-1}} \le \ga \quad \forall k \geq 1.
  \end{equation*}
\end{assumption}
This assumption means that the martingale increment $\ep_k$,
normalized by $\sigma_{k-1}$, is uniformly subgaussian. In the case
where $\varepsilon_k$ is Gaussian conditionally to $\mathscr F_{k-1}$,
Equation (\ref{hypmoments}) is satisfied if $(\sigma_k)$ is such that
$\var(\varepsilon_k | \F_{k-1}) \leq c \sigma_{k-1}^2$ for any $k \geq
0$, where $c > 0$ is a deterministic constant not depending on $k$. If
one assumes that $\var(\varepsilon_k | \F_{k-1}) \leq \bar \sigma ^2$
for a known constant $\bar \sigma > 0$, one can take simply $\sigma_k
\equiv \bar \sigma$. Note that $\sigma_{k-1}$ is not necessarily the
conditional variance of $\varepsilon_k$, but an observed upper bound
of it.

Particular cases of model~\eqref{regmart} are the regression and the
auto-regressive model.
\begin{example}
  In the regression model, one observes $(Y_k, X_{k-1})_{k=1}^n$
  satisfying
  \begin{equation*}
    Y_k = f(X_{k-1}) + s(X_{k-1}) \zeta_k,
  \end{equation*}
  where $(\zeta_k)$ is i.i.d. centered, such that $\E( \exp(\mu
  \zeta_k^2) ) \leq \gamma$ and independent of $\mathscr F_k =
  \sigma(X_0, \ldots, X_k)$, and where $f : \R^d \rightarrow \R$ and
  $s : \R^d \rightarrow \R^+$. This model is a particular case
  of~\eqref{regmart} with $\sigma_k^2 \geq s(X_k)^2$.
\end{example}

\begin{example}
  In the auto-regressive model, one observes a sequence
  $(X_k)_{k=0}^n$ in $\R^d$ satisfying
  \begin{equation}
    \label{eq:autoregression}
    X_k =  \vec{f}(X_{k-1}) + S(X_{k-1}) \vec{\zeta_k},
  \end{equation}
  where $\vec{f} = (f_1, \ldots, f_d) : \R^d \rightarrow \R^d$, where
  $S : \mathbb R^d \rightarrow \mathbb R^{d \times d}$ and where
  $\vec{\zeta_k} = (\zeta_{k,1}, \ldots, \zeta_{k,d})$ is a sequence
  of centered i.i.d. vectors in $\R^d$ independent of $X_0$, with
  covariance matrix $I_d$ and such that $\E( \exp(\mu \zeta_{k, j}^2)
  ) \leq \gamma$. The problem of estimation of each coordinate $f_j$
  is a particular case of~\eqref{regmart} with $Y_k = (X_k)_j$,
  $\mathscr F_k = \sigma(X_0, \vec\zeta_1, \ldots, \vec\zeta_k)$ and
  $\sigma_k^2 \geq S_{j, j}(X_k)^2$.
\end{example}

Let us mention that these two examples are very particular. The
analysis conducted here allows to go way beyond the i.i.d. case, as
long as $(\zeta_k)$ is a martingale increment.

\begin{remark}
  \label{rem:stoppingtime}
  The results given in Section~\ref{sec:main-results} are stated in a
  setting where one observes $(X_{k-1}, Y_k)_{k=1}^N$ with $N$ a
  stopping time. Of course, this contains the usual case $N \equiv n$,
  where $n$ is a fixed sample size. This framework includes situations
  where the statistician decides to stop the sampling according to
  some design of experiment rule. This is the case when obtaining data
  has a cost, that cannot be more than a maximum value, for instance.
\end{remark}

\begin{remark}
  Note that while $\zeta_k = \varepsilon_k / \sigma_{k-1}$ is
  conditionally subgaussian, $\varepsilon_k$ is not in general, (see
  \cite{MR853051} for examples).
\end{remark}

\subsection{The Lepski's method}
\label{sec:lepski_method}

In what follows, $|x|$ stands for the Euclidean norm of $x \in
\R^d$. An object of importance in the analysis conducted below is the
following. For $h > 0$, we define
\begin{equation*}
  L(h) = \sum_{k=1}^N \frac{1}{\si^2_{k-1}} \ind{\abs{X_{k-1} - x} \le
    h},
\end{equation*}
which is the occupation time of $(X_k)_{k \geq 0}$ at $x$ renormalized
by $(\sigma_k)$. Then, if $h$ is such that $L(h) > 0$ (there is at
least one observations in the interval $[x - h, x + h]$), we define
the kernel estimator
\begin{equation*}
  \hat f(h)= \frac{1}{L(h)} \sum_{k=1}^N \frac{1}{\si^2_{k-1}}
  \ind{\abs{X_{i-1} - x}\le h} Y_k.
\end{equation*}
Let $(h_i)_{i \geq 0}$ be a decreasing sequence of positive numbers,
called \emph{bandwidths}, and define the following set, called
\emph{grid}, as
\begin{equation*}
  \HH := \{ h_j : L(h_j) > 0 \}.
\end{equation*}
For the sake of simplicity, we will consider only on a geometrical
grid, where
\begin{equation*}
  h_j = h_0 q^j
\end{equation*}
for some parameters $h_0 > 0$ and $q \in (0, 1)$. The Lepski's method
selects one of the bandwidths in $\mathcal H$. Let $b > 0$ and for any
$h > 0$, define
\begin{equation*}
  \psi(h) := 1 + b \log(h_0 / h).
\end{equation*}
For $u > 0$, define, on the event $\set{ L(h_0)^{-1/2} \le u}$, the
bandwidth
\begin{equation}
  \label{eq:Hu}
  H_u = \min \Big\{ h \in \HH : \Big( \frac{\psi(h)}{L(h)} \Big)^{1/2}
  \le u \Big\},
\end{equation}
and let $u_0>0$. The estimator of $f(x)$ is $\hat f({\hat H})$ defined
on the set $\{ L(h_0)^{-1/2} \le u_0\}$, where $\hat H$ is selected
according to the following rule:
\begin{align}
  \nonumber \hat H := \max \Big\{ h \in \HH : h \ge H_{u_0} \text{ and
  } &\forall
  h'\in[\Hb_{u_0},h] \cap \HH, \\
  &| \hat f(h) - \hat f(h') | \le \nu \Big(\frac{\psi(h')}{L(h')}
  \Big)^{1/2} \Big\},
  \label{eq:hatH}
\end{align}
where $\nu$ is a positive constant. This is the standard Lepski's
procedure, see \cite{lepski90, lepski_92, lepski_mammen_spok_97,
  lepski_spok97}. In the next Section, we give an upper bound for
$\hat f(\hat H)$, with a normalization (convergence rate) that
involves $L(h)$. This result is stated without any further assumptions
on the model.

\begin{remark}
  \label{re:u0}
  The number $u_0$ is a fixed constant such that the largest bandwidth
  $h_0$ in the grid satisfies $L(h_0)^{-1/2} \leq u_0$. This
  deterministic constraint is very mild: if we have some data close to
  $x$, and if $h_0$ is large enough (this is the largest bandwidth in
  the grid), then $L(h_0)$ should be large, at least such that
  $L(h_0)^{-1/2} \leq u_0$. Consider the following basic example: $X_k
  \in [-1, 1]^d$ almost surely for any $k$ and $\sigma_k \equiv 1$,
  then by taking $h_0 = \sqrt d$ and $u_0 = 1$ the event $\{
  L(h_0)^{-1/2} \leq u_0 \}$ has probability one. In
  Section~\ref{sec:examples} (see Proposition~\ref{prop:det-rate}) we
  prove that a mixing assumption on $(X_k)_{k \geq 0}$ entails that
  this event has an overwhelming probability.
\end{remark}

\section{Adaptive upper bound}
\label{sec:main-results}

The usual way of stating an adaptive upper bound for $\hat f(\hat H)$,
see for instance \cite{lepski_mammen_spok_97}, is to prove that it has
the same convergence rate as the \emph{oracle} estimator $\hat
f(H^*)$, which is the ``best'' among a collection $\{ \hat f(h) : h
\in \mathcal H \}$. The oracle bandwidth $H^*$ realizes a
bias-variance trade-off, that involves explicitly the unknown $f$. For
$h \in \HH$ define
\begin{equation}
  \label{eq:ftilde}
  \tilde f(h) := \frac{1}{L(h)} \sum_{k=1}^N \frac{1}{\si^2_{k-1}}
  \ind{\abs{X_{k-1}-x}\le h} f(X_{k-1}).
\end{equation}
Consider a family of non-negative random variables $(W(h) ; h \in
\mathcal H)$ that bounds from above the local smoothness of $f$
(measured by its increments):
\begin{equation}
  \label{eq:bias-assumption}
  \sup_{h' \in [\Hb_{u_0}, h] \cap \HH} \bigabs{\tilde f(h')-
    f(x)} \leq W(h), \quad \forall h \in \mathcal H.
\end{equation}
Nothing is required on $(W(h) : h \in \mathcal H)$ for the moment, one
can perfectly choose it as the left hand side
of~\eqref{eq:bias-assumption} for each $h \in \mathcal H$ for
instance. However, for the analysis conducted here, we need to bound
$W$ from below and above (see Remark~\ref{rem:barW}): introduce
\begin{equation}
  \label{eq:defWbar}
  \bar{W}(h) := [W(h) \vee (\delta_0 (h/h_0)^{\al_0})] \land u_0,
\end{equation}
where $\delta_0$ and $\al_0$ are positive constants. On the set
\begin{equation*}
  \big\{ L(h_0)^{-1/2} \le \bar W(h_0) \big\},
\end{equation*}
define the random \emph{oracle} bandwidth
\begin{equation}
  \label{eq:Hstar}
  H^* := \min \Big\{ h \in \HH : \Big( \frac{\psi(h)}{L(h)} \Big)^{1/2}
    \le \bar W(h) \Big\},
\end{equation}
and consider the event
\begin{equation*}
  \Omega' := \big\{ L(h_0)^{-1/2} \le \bar W(h_0),  W(H^*) \le u_0
  \big\}.
\end{equation*}
The event $\Omega'$ is the ``minimal'' requirement for the proof of an
upper bound for $\hat f(\hat H)$, see Remarks~\ref{rem:barW}
and~\ref{rem:omega} below.
\begin{theorem}
  \label{thm:upper_bound}
  Let Assumption~\ref{hypmoments} hold and let $\hat f(\hat H)$ be the
  procedure given by the Lepski's rule~\eqref{eq:hatH}. Then, for any
  $\rho \in (0, b \mu \nu^2 / (64 \alpha_0 (1 + \gamma)) )$, we have
  \begin{equation*}
    \P \Big[ \Big\{ \abs{ \hat f (\hat H)- f(x) } \geq t \bar W(H^*)
    \Big\} \cap \Omega' \Big] \leq C_0 \frac{(\log(t+1))^{1 + \rho /
        2}}{t^\rho}
  \end{equation*}
  for any $t \geq t_0$, where $C_0, t_0 > 0$ are constants depending
  on $\rho, \mu,\ga, q, b, u_0, \delta_0, \al_0,\nu$.
\end{theorem}

The striking fact in this Theorem is that we don't use any
stationarity, ergodicity or concentration property.  In particular, we
cannot give at this point the behaviour of the random normalization
$\bar W(H^*)$. It does not go to $0$ in probability with $N
\rightarrow +\infty$ when $L(h_0)$ does not go to $+\infty$ in
probability, which happens if $(X_k)_{k \geq 0}$ is a transient Markov
chain for instance. Hence, without any further assumption,
Theorem~\ref{thm:upper_bound} does not entail that $\hat f(\hat H)$ is
close to $f(x)$. On the other hand, when $(X_k)_{k \geq 0}$ is mixing,
we prove that $\bar W(H^*)$ behaves as the deterministic minimax
optimal rate, see Section~\ref{sec:examples}. The cornerstone of the
proof of this Theorem is a new result concerning the stability of
self-normalized martingales, see Theorem~\ref{thm:key} in
Section~\ref{sec:martingale} below.

\begin{remark}
  The parameter $\rho$ of decay of the probability in
  Theorem~\ref{thm:upper_bound} is increasing with the threshold
  parameter $\nu$ from~\eqref{eq:hatH}. So, for any $p > 0$ and $\nu$
  large enough, Theorem~\ref{thm:upper_bound} entails that the
  expectation of $(\bar W(H^*)^{-1} |\hat f(\hat H) - f(x) |)^p
  \ind{\Omega'}$ is finite.
\end{remark}

\begin{remark}
  \label{rem:barW}
  The definition of $\bar W$ is related to the fact that since nothing
  is required on the sequence $(X_k)$, the occupation time $L(h)$ can
  be small, even if $h$ is large. In particular, $L(h)$ has no reason
  to be close to its expectation. So, without the introduction of
  $\bar W$ above, that bounds from below $W$ by a power function, we
  cannot give a lower estimate of $H^*$ (even rough), which is
  mandatory for the proof of Theorem~\ref{thm:upper_bound}.
\end{remark}

\begin{remark}
  \label{rem:omega}
  On the event $\Omega'$, we have $\{ L(h_0)^{-1/2} \leq \bar W(h_0)
  \}$, meaning that the bandwidth $h_0$ (the largest in $\mathcal H$)
  is large enough to contain enough points in $[x - h_0, x + h_0]$, so
  that $L(h_0) \geq \bar W(h_0)^2$. This is not a restriction when
  $W(h) = L h^s$ [$f$ has a local H\"older exponent $s$] for instance,
  see Section~\ref{sec:examples}.
\end{remark}

\begin{remark}
  \label{rem:smoothness1}
  In the definition of $\hat f(\hat H)$, we use kernel estimation with
  the rectangular kernel $K(x) = \ind{[-1, 1]}(x)/2$. This is mainly
  for technical simplicity, since the proof of
  Theorem~\ref{thm:upper_bound} is already technically
  involved. Consequently, Theorem~\ref{thm:upper_bound} does not give,
  on particular cases (see Section~\ref{sec:examples}), the adaptive
  minimax rate of convergence for regression functions with an
  H\"older exponent $s$ larger than $1$. To improve this, one can
  consider the Lepski's method applied to local polynomials (LP) (see
  \cite{MR2339297}, and see \cite{fan_gijbels96} about (LP)). This
  would lead, in the framework considered here, to strong technical
  difficulties.
\end{remark}

\section{Stability for self-normalized martingales}
\label{sec:martingale}

We consider a local martingale $(M_n)_{n \in \mathbb N}$ with respect
to a filtration $(\mathcal G_n)_{n \in \mathbb N}$, and for $n \geq 1$
denote its increment by $\Delta M_n := M_n - M_{n-1}$. The predictable
quadratic variation of $M_n$ is
\begin{equation*}
  \prodsca{M}_n := \sum_{k=1}^n \E[ \Delta M_k^2 | \mathcal G_{k-1} ].
\end{equation*}
Concentration inequalities for martingales have a long history. The
first ones are the Azuma-Hoeffding's inequality (see~\cite{azuma67},
\cite{hoeffding63}) and the Freedman's inequality (see
\cite{freedman75}). The latter states that, if $(M_n)$ is a square
integrable martingale such that $|\Delta M_k| \leq c$ a.s. for some
constant $c > 0$ and $M_0 = 0$, then for any $x, y > 0$:
\begin{equation}
  \label{eq:freeman}
  \P [ M_n \geq x, \prodsca{M}_n \leq y] \leq \exp \Big( -
  \frac{x^2}{2(y + cx)} \Big).
\end{equation}
Later on, an alternative to the assumption $|\Delta M_k| \leq c$ was
proposed. This is the so-called Bernstein's condition, which requires
that there is some constant $c > 0$ such that for any $p \geq 2$:
\begin{equation}
  \label{eq:bernstein_condition}
  \sum_{k=1}^n \E\big[ | \Delta M_k|^p | \mathcal G_{k-1} \big] \leq
  \frac{p!}{2} c^{p-2} \prodsca{M}_n,
\end{equation}
see \cite{delapena99}, and \cite{pinelis94}. In \cite{van_de_geer00}
(see Chapter~8), inequality~\eqref{eq:freeman} is proved with
$\prodsca{M}_n$ replaced by a $\mathcal G_{n-1}$-measurable random
variable $n R_n^2$, under the assumption that
\begin{equation}
  \label{eq:vdg_condition}
  \sum_{k=1}^n \E\big[ | \Delta M_k|^p | \mathcal G_{k-1} \big] \leq
  \frac{p!}{2} c^{p-2} n R_n^2
\end{equation}
holds for any $p \geq 2$. There are many other very recent deviation
inequalities for martingales, in particular inequalities involving the
quadratic variation $[M]_n = \sum_{k=1}^n \Delta M_k^2$, see for
instance~\cite{delapena99} and \cite{MR2462551}.

For the proof of Theorem~\ref{thm:upper_bound}, a Bernstein's type of
inequality is not enough: note that in~\eqref{eq:freeman}, it is
mandatory to work on the event $\{ \prodsca{M}_n \leq y \}$. A control
of the probability of this event usually requires an extra assumption
on $(X_k)_{k \geq 0}$, such as independence or mixing (see
Section~\ref{sec:examples}), and this is precisely what we wanted to
avoid here. Moreover, for the proof of Theorem~\ref{thm:upper_bound},
we need a result concerning $M_T$, where $T$ is an arbitrary finite
stopping-time.

In order to tackle this problem, a first idea is to try to give a
deviation for the self-normalized martingale $M_T / \sqrt{\langle M
  \rangle_T}$.  It is well-known that this is not possible, a very
simple example is given in Remark~\ref{rem:nottight} below. In the
next Theorem~\ref{thm:key}, we give a simple solution to this
problem. Instead of $M_T / \sqrt{\langle M \rangle_T}$, we consider
$\sqrt{a} M_T / (a + \langle M \rangle_T)$, where $a > 0$ is an
arbitrary real number, and we prove that the exponential moments of
this random variable are uniformly bounded under
Assumption~\ref{ass:concentration} below. The result stated in
Theorem~\ref{thm:key} is of independent interest, and we believe that
it can be useful for other statistical problems.
\begin{assumption}
  \label{ass:concentration}
  Assume that $M_0 = 0$ and that
  \begin{equation}
    \label{eq:increment}
    \Delta M_n = s_{n-1} \zeta_n
  \end{equation}
  for any $n \geq 1$, where $(s_n)_{n \in \mathbb N}$ is a $(\mathcal
  G_n)$-adapted sequence of random variables and $(\ze_n)_{n \geq 1}$
  is a sequence of $(\mathcal G_n)$-martingale increments such that
  for $\alpha = 1$ or $\alpha = 2$ and some $\mu > 0, \gamma >
  1$\textup:
  \begin{equation}
    \label{eq:key_prop_ass}
    \E\big[ \exp( \mu |\zeta_k|^\alpha) | \mathcal G_{k-1} \big] \leq
    \gamma \;\; \text{ for any } \;\; k \geq 1.
  \end{equation}
\end{assumption}
Let us define
\begin{equation*}
  V_n := \sum_{k=1}^n s_{k-1}^2.
\end{equation*}
Note that if $(\zeta_n)_{n \geq 1}$ is a conditionally normalized
sequence (ie $\E(\zeta_n^2 | \mathcal G_{n-1}) = 1$)
then~\eqref{eq:increment} entails that $V_n = \langle M
\rangle_n$. Moreover, if Assumption~\ref{ass:concentration} holds, we
have $\langle M \rangle_n \leq c_\mu V_n$ for any $n \geq 1$ with
$c_\mu = \ln 2 / \mu$ when $\alpha = 2$ and $c_\mu = 2 / \mu^2$ when
$\alpha = 1$. Denote $\cosh(x) = (e^x + e^{-x}) / 2$ for any $x \in
\mathbb R$.
\begin{theorem}
  \label{thm:key}
  Let Assumption~\ref{ass:concentration} holds. \\
  $\bullet$~If $\alpha = 2$\textup, we have for any $\la \in [0,
  \frac{\mu}{2(1 + \gamma)})$\textup, any $a>0$ and any finite
  stopping-time $T$\textup:
  \begin{equation}
    \label{eq:prop_psi2}
    \E \Bigcro{\exp\Bigpar{\la \frac{a
          M_T^2}{(a+V_T)^2} }} \le 1 + c_{\lambda},
  \end{equation}
  where $c_{\lambda} := \exp\bigpar{ \frac{\lambda
      \Gamma_{\lambda}}{2(1 - 2\lambda \Gamma_{\lambda})}}
  (\exp(\lambda \Gamma_{\lambda}) - 1)$ and $\Gamma_{\lambda} :=
  \frac{1 + 2\gamma}{2(\mu - \lambda)}.$ \\
  $\bullet$ If $\alpha = 1$\textup, we have for any $\lambda \in
  (-\mu, \mu)$, any $a > 0$ and any finite stopping-time $T$:
  \begin{equation}
    \label{eq:prop_psi1}
    \E\Bigcro{\cosh\Bigpar{\lambda \frac{\sqrt{a}
          M_T}{a+V_T} }} \le 1 + c_{\lambda}',
  \end{equation}
  where $c_{\lambda}' = (\gamma-1) \lambda^2 \exp\bigpar{(\gamma-1)
    \lambda^2 / \mu^2} \cosh\bigpar{2\log 2 + 2 (\gamma-1) \lambda^2 /
    \mu^2} / \mu^2$.
\end{theorem}
The proof of Theorem~\ref{thm:key} is given in
Section~\ref{sec:main_proofs}. Theorem~\ref{thm:key} shows that when
$\zeta_k$ is subgaussian (resp. sub-exponential) conditionally to
$\mathcal G_{k-1}$, then $\sqrt a |M_T| / (a + V_T)$ is also
subgaussian (resp. sub-exponential), hence the name
\emph{stability}. Indeed, we cannot expect an improvement in the tails
of $\sqrt a |M_T| / (a + V_T)$ due to the summation, since the
$s_{k-1}$ are arbitrary (for instance, it can be equal to zero for
every $k$ excepted for one).

\begin{remark}
  \label{rem:nottight}
  It is tempting to take ``$a = V_T$'' in
  Theorem~\ref{thm:key}. However, the following basic example shows
  that it is not possible. Take $(B_t)_{t \geq 0}$ a standard Brownian
  motion, consider $M_n = B_n$ and define the stopping time $T_c =
  \inf \{ n \geq 1 : B_n / \sqrt n \geq c \}$, where $c > 0$. For any
  $c > 0$, $T_c$ is finite a.s. (use the law of iterated logarithm for
  instance). So, in this example, one has $M_{T_c} / \sqrt{\langle M
    \rangle_{T_c}} = M_{T_c} / \sqrt{T_c} \geq c$, for any $c > 0$.
\end{remark}

\section{Consistency with the minimax theory of deterministic rates}
\label{sec:examples}

In this Section, we prove that, when $(X_k)_{k \geq 0}$ is mixing,
then Theorem~\ref{thm:upper_bound} gives the adaptive minimax upper
bound. Let us consider again sequences $(X_k)_{k\ge 0}$ and
$(Y_k)_{k\ge 1}$ of random variables satisfying \eqref{regmart}, where
$(\varepsilon_k)_{k \geq 0}$ an $(\F_k)_{k\ge 0}$-martingale
increment. For the sake of simplicity, we work under the following
simplified version of Assumption~\ref{hypmoments}.
\begin{assumption}
  \label{ass:hypmomentssimplified}
  There is a known $\sigma > 0$ and $\mu, \gamma > 0$ such
  that\textup:
  \begin{equation*}
    \E\Bigcro{\exp\Bigpar{\mu \frac{\ep_k^2} {\si^2} } \mid
      \F_{k-1}} \le \ga \quad \forall k \geq 1.
  \end{equation*}
\end{assumption}
Moreover, we consider the setting where we observe $(Y_1,\dots, Y_n)$
and $(X_0,\dots, X_{n-1})$, namely the stopping-time $N$ is simply
equal to $n$ (the results in this section are proved for $n$ large
enough). Note that in this setting, we have $L(h) = \sigma^{-2}
\sum_{k=1}^n \ind{|x_{k-1}-x| \leq h}$. We assume also that
$(X_k)_{k \geq 0}$ is a strictly stationary sequence.

\subsection{Some preliminaries}

A function $\ell : \mathbb R^+ \rightarrow \mathbb R^+ $ is slowly
varying if it is continuous and if
\begin{equation*}
  \lim_{h \rightarrow 0^+} \ell(y h) / \ell(h) = 1, \quad \forall y >
  0.
\end{equation*}
Fix $\tau \in \mathbb R$. A function $g : \mathbb R^+ \rightarrow
\mathbb R^+$ is $\tau$-regularly varying if $g(y) = y^\tau
\ell(y)$ for some slowly varying~$\ell$. Regular variation is a
standard and useful notion, of importance in extreme values theory for
instance. We refer to \cite{bgt89} on this topic.

Below we will use the notion of $\beta$-mixing to measure the
dependence of the sequence $(X_k)_{k \geq 0}$. This measure of
dependence was introduced by Kolmogorov, see \cite{MR0133175}, and we
refer to \cite{MR1312160} for topics on dependence. Introduce the
$\sigma$-field $\mathscr X_{u}^v = \sigma(X_k : u \leq k \leq v)$,
where $u, k, v$ are integers. A strictly stationary process $(X_k)_{k
  \in \mathbb Z}$ is called $\beta$-mixing or absolutely regular if
\begin{equation}
  \label{eq:beta_coef}
  \beta_q := \frac{1}{2} \sup \Big( \sum_{i=1}^I \sum_{j=1}^J \Big| \P[U_i \cap
  V_j] - \P[U_i] \P[V_j] \Big| \Big) \rightarrow 0 \text{ as } q \rightarrow
  +\infty,
\end{equation}
where the supremum is taken among all finite partitions
$(U_i)_{i=1}^I$ and $(V_j)_{j=1}^J$ of $\Omega$ that are, respectively,
$\mathscr X_{-\infty}^0$ and $\mathscr X_{q}^{+\infty}$
measurable. This notion of dependence is convenient in statistics
because of a coupling result by Berbee, see \cite{MR547109}, that
allows to construct, among $\beta$-mixing observations, independent
blocks, on which one can use Bernstein's or Talagrand's inequality
(for a supremum) for instance. This strategy has been adopted in a
series of papers dealing with dependent data, see \cite{MR1440142,
  MR1865343, MR1614587} among others. In this section, we use this
approach to give a deterministic equivalent to the random rate used in
Section~\ref{sec:main-results}. This allows to prove that
Theorem~\ref{thm:upper_bound} is consistent with the usual minimax
theory of deterministic rates, when one assumes that the sequence
$(X_k)_{k \geq 0}$ is $\beta$-mixing.

\subsection{Deterministic rates}

We assume that $f$ has H\"older-type smoothness in a neighbourhood of
$x$. Let us fix two constants $\delta_0, u_0 > 0$ and recall that
$h_0$ is the maximum bandwidth used in the Lepski's procedure (see
Section~\ref{sec:lepski_method}).

\begin{assumption}[Smoothness of $f$]
  \label{ass:fsmoothness}
  There is $0 < s \leq 1$ and a slowly varying function $\ell_w$ such
  that the following holds:
  \begin{equation*}
    \sup_{y : |y - x| \leq h} |f(y) - f(x)| \leq w(h), \text{ where }
    w(h):= h^s \ell_w(h)
  \end{equation*}
  for any $h \leq h_0$, $w$ is increasing on $[0, h_0]$, $w(h) \geq
  \delta_0 (h / h_0)^2$ and $w(h) \leq u_0$ for any $h \in [0,
  h_0]$.
\end{assumption}
This is slightly more general than an H\"older assumption because of
the slowly varying term $\ell_w$. The usual H\"older assumption is
recovered by taking $\ell_w \equiv r$, where $r > 0$ is some constant
(the radius in H\"older smoothness). 

Under Assumption~\ref{ass:fsmoothness}, one has that
\begin{equation*}
  \sup_{h' \in [H_{u_0}, h] \cap \mathcal H} |\tilde f(h') - f(x)|
  \leq w(h) \quad \forall h \in \mathcal H.
\end{equation*}
Under this assumption, one can replace $\bar W$ by $w$ in the
statement of Theorem~\ref{thm:upper_bound} and from the definition of
the oracle bandwidth $H^*$ (see~\eqref{eq:Hstar}). An oracle bandwidth
related to the modulus of continuity $w$ can be defined in the
following way: on the event
\begin{equation*}
  \Omega_0 = \{ L(h_0)^{-1/2} \leq w(h_0) \},
\end{equation*}
let us define
\begin{equation}
  \label{eq:Hw}
  H_w := \min \Big\{ h \in ]0, h_0]  : \Big( \frac{\psi(h)}{L(h)}
  \Big)^{1/2} \le w(h) \Big\}.
\end{equation}
Under some ergodicity condition (using $\beta$-mixing) on $(X_k)_{k
  \geq 0}$, we are able to give a deterministic equivalent to
$w(H_w)$. Indeed, in this situation, the occupation time $L(h)$
concentrates around its expectation $\E L(h)$, so a natural
deterministic equivalent to~\eqref{eq:Hw} is given by
\begin{equation}
  \label{eq:hw}
  h_w := \min \Big\{ h \in ]0, h_0] : \Big( \frac{\psi(h)}{\E L(h)}
  \Big)^{1/2} \le w(h) \Big\}.
\end{equation}
Note that $h_w$ is well defined and unique when $(\E L(h_0))^{-1/2}
\leq w(h_0)$, ie when $n \geq \sigma^2 / (P_X([x-h_0, x+h_0])
w(h_0)^2)$, where $P_X$ stands for the distribution of $X_0$. We are
able to give the behaviour of $h_w$ under the following assumption.

\begin{assumption}[Local behaviour of $P_X$]
  \label{ass:behavior-PX}
  There is $\tau \geq -1$ and a slowly varying function $\ell_X$
  such that 
  \begin{equation*}
    P_X( [x - h, x + h] ) = h^{\tau + 1} \ell_X(h) \quad \forall h
    \leq h_0.
  \end{equation*}
\end{assumption}
This is an extension of the usual assumption on $P_X$ which requires
that it has a continuous density $f_X$ wrt the Lebesgue measure such
that $f_X(x) > 0$ (see also \cite{MR2339297}). It is met when $f_X(y)
= c |y - x|^\tau$ for $y$ close to $x$ for instance (in this case
$\ell_X$ is constant).

\begin{lemma}
  \label{lem:hw}
  Grant Assumptions~\ref{ass:fsmoothness}
  and~\ref{ass:behavior-PX}. Then $h_w$ is well defined
  by~\eqref{eq:hw} and unique when $n$ is large enough and such that
  \begin{equation*}
    h_w = (\sigma^2 / n)^{1 / (2s + \tau + 1)} \ell_1(\sigma^2 / n)
    \text{ and } w(h_w) = (\sigma^2 / n)^{s / (2s + \tau + 1)}
    \ell_2(\sigma^2 / n),
  \end{equation*}
  where $\ell_1$ and $\ell_2$ are slowly varying functions that depend
  on $s, \tau$ and $\ell_X$, $\ell_w$.
\end{lemma}

The proof of this lemma easily follows from basic properties of
regularly varying functions, so it is omitted. Explicit examples of
such rates are given in \cite{MR2339297}. Note that in the
i.i.d. regression setting, we know from \cite{MR2339297} that $w(h_w)$
is the minimax adaptive rate of convergence. Now, under the following
mixing assumption, we can prove that the random rate $w(H_w)$ and the
deterministic rate $w(h_w)$ have the same order of magnitude with a
large probability.

\begin{assumption}
  \label{ass:Xbeta-mixing}
  Let $(\beta_q)_{q \geq 1}$ be the sequence of $\beta$-mixing
  coefficients of $(X_k)_{k \geq 0}$, see~\eqref{eq:beta_coef}, and
  let $\eta, \kappa > 0$. We assume that for any $q \geq 1$:
  \begin{equation*}
    \beta_q \leq \frac{1}{\psi^{-1}(2 q)},
  \end{equation*}
  where $\psi(u) = \eta (\log u)^\kappa$ (geometric mixing) or
  $\psi(u) = \eta u^\kappa$ (arithmetic mixing).
\end{assumption}

\begin{prop}
  \label{prop:det-rate}
  Let Assumptions~\ref{ass:fsmoothness}, \ref{ass:behavior-PX}
  and~\ref{ass:Xbeta-mixing} hold. On $\Omega_0$, let $H_w$ be given
  by~\eqref{eq:Hw} and let (for $n$ large enough) $h_w$ be given
  by~\eqref{eq:hw}. Then, if $(X_k)$ is geometrically $\beta$-mixing,
  or if it is arithmetically $\beta$-mixing with a constant $\kappa <
  2s / (\tau + 1)$, we have
  \begin{align*}
    \P \Big[ \Big\{ \frac{w(h_w)}{4} \leq w(H_w) \leq 4 w(h_w) \Big\}
    \cap \Omega_0 \Big] \geq 1 - \phi_n \;\; \text{ and } \;\; \P [
    \Omega_0^\complement ] = o(\varphi_n)
  \end{align*}
  for $n$ large enough, where in the geometrically $\beta$-mixing
  case:
  \begin{equation*}
    \phi_n = \exp( -C_1 n^{\delta_1} \ell_1(1 / n))
    \text{ where } \delta_1 = \frac{2 s}{(2s + \tau + 1) (\kappa+1)}
  \end{equation*}
  and in the arithmetically $\beta$-mixing case:
  \begin{equation*}
    \phi_n = C_2 n^{-\delta_2} \ell_2(1/n) \text{ where } \delta_2 =
    \frac{2s}{2s + \tau +1} \Big( \frac1\kappa - \frac{\tau + 1}{2s}
    \Big),
  \end{equation*}
  where $C_1, C_2$ are positive constants and $\ell_1, \ell_2$ are
  slowly varying functions that depends on $\eta, \kappa, \tau, s,
  \sigma$ and $\ell_X$, $\ell_w$.
\end{prop}

The proof of Proposition~\ref{prop:det-rate} is given in
Section~\ref{sec:main_proofs} below. The assumption used in
Proposition~\ref{prop:det-rate} allows a geometric $\beta$-mixing, or
an arithmetic $\beta$-mixing, up to a certain order, for the sequence
$(X_k)$. This kind of restriction on the coefficient of arithmetic
mixing is standard, see for instance \cite{MR1614587,MR1440142,
  MR1865343}.

The next result is a direct corollary of Theorem~\ref{thm:upper_bound}
and Proposition~\ref{prop:det-rate}. It says that when $(X_k)_{k \geq
  0}$ is mixing, then the deterministic rate $w(h_w)$ is an upper
bound for the risk of $\hat f(\hat H)$. 

\begin{corollary}
  \label{cor:det-upper-bound}
  Let Assumptions~\ref{ass:hypmomentssimplified},
  \ref{ass:fsmoothness} and \ref{ass:behavior-PX} hold. Let
  Assumption~\ref{ass:Xbeta-mixing} hold, with the extra assumption
  that $\kappa < 2s / (s + \tau +1)$ in the arithmetical
  $\beta$-mixing case. Moreover, assume that $|f(x)| \leq Q$ for some
  known constant $Q > 0$. Let us fix $p > 0$.  If $\nu > 0$ satisfies
  $b \mu \nu^2 > 128 p (1 + \tau) $ (recall that $\nu$ is the constant
  in front the threshold in the Lepski's procedure,
  see~\eqref{eq:hatH}) then we have
  \begin{equation*}
    \E [ |\tilde f (\hat H) - f(x)|^p ] \leq C_1 w(h_w)^p
  \end{equation*}
  for $n$ large enough, where $\tilde f (\hat H) = -Q \vee \hat f
  (\hat H) \wedge Q$ and where $C_1 > 0$ depends on $q, p, s, \mu,
  \ga, b, u_0, \delta_0, \nu, Q$.
\end{corollary}

The proof of Corollary~\ref{cor:det-upper-bound} is given in
Section~\ref{sec:main_proofs} below. Let us recall that in the i.i.d
regression model with gaussian noise, we know from~\cite{MR2339297}
that $w(h_w)$ is the minimax adaptive rate of convergence. So,
Corollary~\ref{cor:det-upper-bound} proves that
Theorem~\ref{thm:upper_bound} is consistent with the minimax theory of
deterministic rates, when $(X_k)$ is $\beta$-mixing.

\begin{example}
  Assume that $f$ is $s$-H\"older, ie Assumption~\ref{ass:fsmoothness}
  holds with $w(h) = L h^s$ so $\ell_w(h) \equiv L$ and assume that
  $P_X$ has a density $f_X$ which is continuous and bounded away from
  zero on $[x - h_0, x + h_0]$, so that
  Assumption~\ref{ass:behavior-PX} is satisfied with $\tau = 0$. In
  this setting, one easily obtains that $w(h_w)$ is equal (up to some
  constant) to $(\log n / n)^{s / (2s + 1)}$, which is the pointwise
  minimax adaptive rate of convergence, see
  \cite{lepski_spok97,lepski_92,lepski_mammen_spok_97} for the
  white-noise model and \cite{MR2339297} for the regression model.
\end{example}

\section{Proof of the main results}
\label{sec:main_proofs}

\subsection{Proof of Theorem~\ref{thm:key} for $\alpha = 2$}

Let $a > 0$ and $\la \in [0, \frac{\mu}{2(1 + \gamma)} )$. Define $Y_0
:= 0$ and for $n \geq 1$:
\begin{equation*}
  Y_n := \frac{a M_n^2}{(a + V_n)^2} \;\text{ and }\;
  H_n := \E\bigcro{\exp\bigpar{\la(Y_{n} - Y_{n-1})} \mid
    \G_{n-1}}.
\end{equation*}
Assume for the moment that $H_n$ is finite a.s, hence we can define
the local martingale
\begin{equation*}
  S_n := \sum_{k = 1}^n e^{\la Y_{k-1}} \bigpar{e^{\la
      (Y_k-Y_{k-1})}-H_k},
\end{equation*}
so that
\begin{align*}
  \exp(\la Y_n) &= 1 + \sum_{k=1}^n e^{\la Y_{k-1}}
  \bigpar{e^{\la (Y_k-Y_{k-1})}-1}  \\
  &=1 + S_n + \sum_{k=1}^n e^{\la Y_{k-1}} (H_k-1).
\end{align*}
Using the sequence of localizing stopping times
\begin{equation*}
  T_p := \min \Big\{n \geq 0 : \sum_{k=1}^{n+1} \E( e^{\lambda Y_k}
  | \mathcal G_{k-1} ) > p \Big\}
  \end{equation*}
  for $p > 0$, the process $(S_{n \wedge T_p})_{n \geq 0}$ is a
  uniformly integrable martingale. So using Fatou's Lemma, one easily
  gets that
  \begin{align*}
    \E( e^{\lambda Y_T} ) \leq \liminf_{p \rightarrow +\infty} \E(
    e^{\lambda Y_{T \wedge T_p}}) &\leq \liminf_{p \rightarrow
      +\infty} \Big\{ 1 + \E( S_{T \wedge T_p}) + \E\Big( \sum_{k=1}^{T
      \wedge T_p} e^{\la Y_{k-1}} (H_k - 1) \Big) \Big\} \\
    &= 1 + \liminf_{p \rightarrow +\infty} \E\Big( \sum_{k=1}^{T
      \wedge T_p} e^{\la Y_{k-1}} (H_k - 1) \Big).
  \end{align*}
  This entails (\ref{eq:prop_psi2}) if we prove that
  \begin{equation}
    \label{majf}
    \sum_{i=1}^n e^{\la Y_{k-1}} (H_k-1) \leq c_{\lambda}
  \end{equation}
  for all $n \ge 1$. First, we prove that
  \begin{equation}
    \label{majH} 
    H_n \le \exp\Bigcro{ \frac{\la a
        s^2_{n-1}} {(a+V_{n})^2} \Bigpar{ \Ga_{\la} + \frac{2
          M_{n-1}^2}{a+V_{n-1}} (2\la\Ga_\la -1) } },
  \end{equation}
  which entails that $H_n$ is finite almost surely.  We can write
  \begin{align*}
    Y_n-Y_{n-1} &= a \frac{M_n^2-M_{n-1}^2}{\bigpar{a + V_n}^2} + a
    M_{n-1}^2 \frac{ \bigpar{a + V_{n-1}}^2-\bigpar{a + V_{n}}^2 }
    { \bigpar{a + V_n}^2 \bigpar{a + V_{n-1}}^2} \\
    &= a \frac{\bigpar{M_n-M_{n-1}}^2 + 2
      M_{n-1}\bigpar{M_n-M_{n-1}}}{\bigpar{a + V_n}^2} \\
    &\quad \quad - \frac{ a M_{n-1}^2 s_{n-1}^2 (2a + V_{n-1}+ V_{n}) }
    { \bigpar{a + V_n}^2 \bigpar{a + V_{n-1}}^2} \\
    &\le \frac{a\bigpar{ s_{n-1}^2 \ze_n^2 + 2M_{n-1}s_{n-1}\ze_n}}
    {\bigpar{a + V_n}^2} -\frac{2aM_{n-1}^2 s_{n-1}^2} { \bigpar{a +
        V_n}^2 \bigpar{a + V_{n-1}}}
  \end{align*}
  where we used that $V_{n-1}\le V_{n}$. In other words
  \begin{equation*}
    \exp\bigpar{\la(Y_n-Y_{n-1})} \le \exp\big( \mu_n \ze_n^2 +
    \rho_n \ze_n - \de_n \big),
  \end{equation*}
  with:
  \begin{equation*}
    \mu_n=\frac{\la a s_{n-1}^2}{\bigpar{a+ V_n}^2}, \quad
    \rho_n= \frac{2\la a s_{n-1} M_{n-1}}{\bigpar{a+ V_n}^2}, \quad
    \de_n=\frac{2\la a s^2_{n-1} M_{n-1}^2}{\bigpar{a + V_n}^2 \bigpar{a
        + V_{n-1}}}.
  \end{equation*}
  The random variables $\mu_n$, $\rho_n$ and $\de_n$ are
  $\G_{n-1}$-measurable and one has $0 \le \mu_n \le \la$.  We need
  the following Lemma.
  \begin{lemma}
    \label{l1} 
    Let $\ze$ be a real random variable such that $\E[\ze] = 0$ and
    such that
    \begin{equation*}
      \mathbb E[ \exp( \mu \zeta^2 ) ] \leq \gamma
    \end{equation*}
    for some $\mu > 0$ and $\gamma > 1$. Then, for any $\rho \in
    \mathbb R$ and $m \in [0, \mu)$\textup, we have
    \begin{equation*}
      \E[e^{m \ze^2 + \rho\ze}] \le \exp \Big( \frac{(1 + 2
        \gamma)(\rho^2 + m) }{2(\mu - m)} \Big).
    \end{equation*}
  \end{lemma}
  The proof of this Lemma is given in
  Section~\ref{sec:lemmas}. Conditionally to $\G_{n-1}$, we apply
  Lemma~\ref{l1} to $\ze_n$. This gives
  \begin{equation*}
    H_n \le \E[\exp(\mu_n \ze_n^2 + \rho_n\ze_n - \de_n) \mid
    \G_{n-1}] \le \exp\Bigpar{ \Ga_\la \bigpar{\rho_n^2 + \mu_n} -
      \de_n},
  \end{equation*}
  that can be be written
  $$ H_n
  \le \exp\Bigcro{ \frac{\la a s^2_{n-1}}{(a+V_{n})^2}
    \Bigpar{\Ga_{\la} + 2M_{n-1}^2 \Bigpar{\frac{2\la \Ga_{\la}
          a}{(a+V_{n})^2}-\frac{1}{a+V_{n-1}}} } }$$ which yields
  \eqref{majH} using $a / (a+V_n)^2 \leq 1/(a+V_{n-1})$. Since
  $\lambda < \mu / [2 (1 + \gamma)]$, we have $2 \lambda
  \Gamma_\lambda - 1 < 0$, so (\ref{majH}) entails
  \begin{equation*}
    H_n-1 \le \exp\Bigcro{
      \frac{\la\Ga_{\la} a s^2_{n-1}} {(a+V_{n})^2}}-1 \le
    (\exp{(\la\Ga_\la)} - 1) \frac{a s^2_{n-1}}{(a+V_n)^2},
  \end{equation*}
  where we used the fact that $e^{\mu x} - 1 \leq (e^\mu - 1) x$ for
  any $x \in [0, 1/2]$, and $\mu > 0$. Note that~\eqref{majH} entails
  also the following inclusion:
  \begin{equation*}
    \set{H_n>1} \subset \set{\frac{2 M_{n-1}^2}{a+V_{n-1}} <
      \frac{\Ga_{\la} }{1-2\la\Ga_\la}} \subset\set{ e^{\la Y_{n-1}} <
      \exp\Big(\frac{\la\Ga_\la}{2(1-2\la\Ga_\la)} \Big) }.
  \end{equation*}
  It follows that
  \begin{equation*}
    \sum_{k = 1}^n e^{\la Y_{k-1}} (H_k - 1) \leq c_{\lambda}
    \sum_{k=1}^n \frac{a s^2_{k-1}}{(a+V_k)^2},
  \end{equation*}
  so \eqref{majf} follows, since
  \begin{equation*}
    \sum_{k=1}^n \frac{a s^2_{k-1}}{(a+V_k)^2} \le \int_0^{V_n}
    \frac{a}{(a+x)^2} dx \le 1.
  \end{equation*}
  This concludes the proof of \eqref{eq:prop_psi2} for $\alpha =
  2$. $\hfill \square$

  \subsection{Proof of Theorem~\ref{thm:key} for $\alpha = 1$}

  First, note that~\eqref{eq:key_prop_ass} and the fact that the
  $\zeta_k$ are centered entails that for any $|\lambda| < \mu$, we
  have
  \begin{equation}
    \label{laplace}
    \E[\exp(\lambda \ze_k) \mid \G_{k-1}] \le \exp(\mu' \lambda^2)
  \end{equation}
  for any $k \geq 1$, where $\mu' = (\gamma-1) / \mu^2$. Now, we use
  the same mechanism of proof as for the case $\alpha = 2$. Let $a >
  0$ and $\lambda \in (-\mu,\mu)$ be fixed. Define
  \begin{equation*}
    Y_n= \frac{\sqrt a M_n}{a + V_n} \text{ and } H_n = \E
    \bigcro{\cosh(\lambda Y_n)-\cosh(\lambda Y_{n-1}) \mid \G_{n-1}}.
  \end{equation*}
  Assuming for the moment that $H_n$ is finite almost surely, we
  define the local martingale
  \begin{equation*}
    S_n := \sum_{k=1}^n \Bigpar{ \cosh(\lambda Y_k) - \cosh(\lambda
      Y_{k-1}) - H_k}.
  \end{equation*}
  Thus, inequality~\eqref{eq:prop_psi1} follows if we prove that for
  all $n\ge 1$:
  \begin{equation*}
    \cosh(\lambda Y_n) \le 1 + S_n + \mu' \lambda^2
    \exp\bigpar{\mu'\lambda^2} \cosh\bigpar{2\log 2 +
      2\mu'\lambda^2}.
  \end{equation*}
  We can write
  \begin{equation*}
    Y_n - Y_{n-1} = -\frac{\sqrt{a}M_{n-1} s^2_{n-1}}{(a+V_n)(a+V_{n-1})}
    + \frac{\sqrt{a} s_{n-1} \ze_n}{a+V_n},
  \end{equation*}
  which gives, together with~\eqref{laplace}:
  \begin{equation*}
    \E\Bigcro{\exp\bigpar{\pm\lambda(Y_n - Y_{n-1})} \mid \G_{n-1}}
    \le \exp\Bigpar{ \pm \frac{\lambda\sqrt{a} M_{n-1}
        s^2_{n-1}}{(a+V_n)(a+V_{n-1})} +\frac{\mu' \lambda^2 a 
        s^2_{n-1}}{(a+V_n)^2} }.
  \end{equation*}
  As we have
  \begin{equation*}
    \cosh(\lambda Y_n) = \frac{1}{2} e^{\lambda Y_{n-1}}
    e^{\lambda(Y_n-Y_{n-1})} + \frac{1}{2} e^{-\lambda Y_{n-1}}
    e^{-\lambda(Y_n-Y_{n-1})},
  \end{equation*}
  we derive:
  \begin{align*}
    \E\Bigcro{\cosh(\lambda Y_n)\mid \G_{n-1}} &\le \frac{1}{2}
    \exp\Bigpar{ \lambda Y_{n-1} -\frac{\lambda\sqrt{a}M_{n-1}
        s^2_{n-1}}{(a+V_n)(a+V_{n-1})} +\frac{\mu' \lambda^2 a
        s^2_{n-1}}{(a+V_n)^2}
    }\\
    &+ \frac{1}{2} \exp\Bigpar{-\lambda Y_{n-1}+
      \frac{\lambda\sqrt{a}M_{n-1} s^2_{n-1}}{(a+V_n)(a+V_{n-1})}
      +\frac{\mu' \lambda^2 a s^2_{n-1}}{(a+V_n)^2}
    },\\
    &= \exp\Bigpar{\frac{\mu' \lambda^2 a s^2_{n-1}}{(a+V_n)^2}}
    \cosh\Bigpar{(1-\frac{s^2_{n-1}}{a+V_n}) \lambda Y_{n-1}}.
  \end{align*}
  So, it remains to prove that
  \begin{align*}
    \sum_{k=1}^n \Big( \exp\Bigpar{\frac{\mu' \lambda^2 a
        s^2_{k-1}}{(a+V_k)^2}} &\cosh \Bigpar{(1 - \frac{s^2_{k-1}}{a
        + V_k}) \lambda Y_{k-1}} - \cosh(\lambda Y_{k-1}) \Big) \\
    &\leq \mu' \lambda^2 \exp\bigpar{\mu'\lambda^2} \cosh\bigpar{2\log 2
      + 2\mu'\lambda^2}.
  \end{align*}
  We need the following lemma.
  \begin{lemma}
    \label{calcul}
    If $A > 0$, one has
    \begin{equation*}
      \sup_{\eta \in [0,1]} \sup_{z \geq 0} \bigpar{
        e^{A\eta}\cosh((1-\eta)z)-\cosh(z)}
      \le A\eta e^{A\eta} \cosh(2\log2+2A).
    \end{equation*}    
  \end{lemma}
  The proof of this Lemma is given in Section~\ref{sec:lemmas}. Using
  Lemma~\ref{calcul} with $\eta=s^2_{k-1}/(a+V_k)$ and $A = \mu'
  \lambda^2 a /(a+V_k)$, we obtain
  \begin{align*}
    \exp\Bigpar{\frac{\mu' \lambda^2 a s^2_{k-1}}{(a + V_k)^2}}
    \cosh\Bigpar{ \big(1 - \frac{s^2_{k-1}}{a+V_k} \big) \lambda
      Y_{k-1}} - \cosh(\lambda Y_{k-1}) \\
    \leq \frac{\mu' \lambda^2 a s^2_{k-1}}{(a + V_k)^2} e^{\mu'
      \lambda^2 } \cosh\bigpar{2\log2 + 2 \lambda^2 \mu'},
  \end{align*}
  and (\ref{eq:prop_psi1}) follows, since
  \begin{equation*}
    \sum_{k=1}^n \frac{a s^2_{k-1}}{(a + V_k)^2} \le \int_0^{V_n}
    \frac{a}{(a + x)^2} dx \le 1. \qedhere
  \end{equation*}
  This concludes the proof of Theorem~\ref{thm:key}.  $\hfill \square$

  \subsection{Proof of Theorem~\ref{thm:upper_bound}}

\subsubsection{Notations}

Let us fix $\lambda \in (0, \frac{\mu}{2(1 + \gamma)})$, to be chosen
later. In the following we denote by $C$ any constant which depends
only on $(\la,\mu,\ga)$. Let us recall that on the event
\begin{equation*}
  \Om' := \{ L(h_0)^{-1/2} \le \bar W(h_0) \} \cap \{ W(H^*) \leq
  u_0 \},
\end{equation*}
the bandwidths $H^*$ and $\hat H$ are well defined, and let us we set
for short
\begin{equation*}
  \P'(A) = \P(\Om'\cap A).
\end{equation*}
We use the following notations: for $h>0$ and $a>0$, take
\begin{equation}
  \label{eq:MZdef}
  M(h) := \sum_{k=1}^N \frac{1}{\si^2_{k-1}} \ind{\abs{X_{k-1}-x}\le h}
  \ep_k, \quad Z(h,a) := \frac{\sqrt{a}\, \abs{M(h)}}{a+L(h)}.
\end{equation}
If $h = h_j \in \mathcal H$, we denote $h_{-} := h_{j+1}$ and $h_+ :=
h_{j-1}$ if $j\ge 1$. We will use repeatedly the following quantity:
for $i_0\in\N$ and $t>0$, consider
\begin{equation}
  \label{eq:pidef}
  \pi(i_0,t) := \P \Big[ \sup_{i\ge i_0} \psi^{-1/2}(h_i)
  \sup_{a \in I(h_i)} Z\bigpar{h_i, a \psi(h_i)} >t \Big],
\end{equation}
where
\begin{equation*}
  I(h) := [u_0^{-2}, \delta_0^{-2} (h / h_0)^{-2\al_0}].
\end{equation*}
Note that this interval is related to the definition of $\bar W$,
see~\eqref{eq:defWbar}. The proof of Theorem~\ref{thm:upper_bound}
contains three main steps. Namely,
\begin{enumerate}
\item the study of the risk of the ideal estimator $\bar W(H^*)^{-1} |
  \hat f(H^*) - f(x)|$,
\item the study of the risk $\bar W(H^*)^{-1} | \hat f(\hat H) -
  f(x)|$ when $\{ H^* \leq \hat H \}$,
\item the study of the risk $\bar W(H^*)^{-1} | \hat f(\hat H) -
  f(x)|$ when $\{ H^* > \hat H \}$.
\end{enumerate}
These are the usual steps in the study of the Lepski's method, see
\cite{lepski90, lepski_92, lepski_mammen_spok_97,
  lepski_spok97}. However, the context (and consequently the proof)
proposed here differs significantly from the ``usual'' proof.

\subsubsection{On the event $\{ H^* \leq \hat H \}$}

Recall that $\nu > 0$ is the constant in front of the Lepski's
threshold, see~\eqref{eq:hatH}. Let us prove the following.
\begin{lemma}
  \label{lem:Hstar-smaller-hatH}
  For all $t > 0$ one has
  \begin{equation}
    \label{f(H^*)}
    \P' \Big[ \bar W(H^*)^{-1} \abs{ \hat f(H^*)-f(x)
    } > t \Big] \le \pi(0, (t-1)/2),
  \end{equation}
  and
  \begin{equation}
    \label{eq:HstarhatH}
    \P' \Big[ H^* \le \hat H, \bar W(H^*)^{-1}
  \bigabs{\hat f(\hat H) - f(x)} >t \Big]
  \le \pi(0, (t-\nu-1)/2).
\end{equation}
\end{lemma}

\begin{proof}
  First, use the decomposition
  \begin{equation*}
    \abs{ \hat f(H^*)-f(x) } \le \abs{ \tilde f(H^*) - f(x) } + \frac{
      \abs{M(H^*)}}{L(H^*)},
  \end{equation*}
  where we recall that $\tilde f(h)$ is given by~\eqref{eq:ftilde},
  and the fact that $\abs{ \tilde f(H^*)-f(x) } \le \bar W(H^*)$,
  since $W(H^*) \leq \bar W(H^*)$ on $\set{W(H^*) \le u_0}$. Then,
  use~\eqref{eq:Hstar} to obtain $L(H^*)^{1/2} \geq \psi(H^*)^{1/2}
  \bar W(H^*)^{-1}$, so that
  \begin{align*}
    \frac{\abs{M(H^*)}}{L(H^*)} &\le \frac{2 \abs{M(H^*)}}{L(H^*) +
      \psi(H^*) \bar W(H^*)^{-2}} \\
    &\le 2 \bar W(H^*) \psi^{-1/2}(H^*) Z\bigpar{H^*, \bar W^{-2}(H^*)
      \psi(H^*)},
  \end{align*}
  and
  \begin{align}
    \bar W^{-1}(H^*) \frac{\abs{M(H^*)}}{L(H^*)} &\le 2
    \psi^{-1/2}(H^*) \sup_{a \in I(H^*)} Z\bigpar{H^*, a \psi(H^*)}
    \nonumber \\
    &\label{Mh} \le 2 \sup_{j \ge 0} \psi^{-1/2}(h_j) \sup_{a \in
      I(h_j)} Z\bigpar{h_j, a \psi(h_j)},
  \end{align}
  this concludes the proof of \eqref{f(H^*)}. On $\{H^*\le \hat H\}$,
  one has using~\eqref{eq:hatH} and~\eqref{eq:Hstar}:
  \begin{equation*}
    | \hat f(\hat H)-\hat f(H^*) | \le \nu ( \psi(H^*) / L(H^*) )^{1/2}
    \le \nu \bar W(H^*).
  \end{equation*}
  Hence, since $W(H^*) \le \bar W(H^*)$ on $\set{W(H^*)\le u_0}$, we
  have for all $t>0$:
  \begin{align*}
    & \P' \Big[ H^* \le \hat H, \bar W(H^*)^{-1}\bigabs{\hat f(\hat
      H)-f(x)}>t \Big] \\
    & \le \P' \Big[ H^* \le \hat H, \bar W(H^*)^{-1}\bigabs{\hat f(
      H^*)-f(x)}>t-\nu \Big],
  \end{align*}
  and~\eqref{eq:HstarhatH} follows using~\eqref{f(H^*)}.
\end{proof}

\subsubsection{On the event $\{ H^* > \hat H \}$}

\begin{lemma}
  \label{majsup}
  For any $t, \eta>0$, we have
  \begin{equation*}
    \P' \Bigpar{ H^*\le \eta, \sup_{ \Hb_{u_0} \le h < H^*, h\in\HH}
      \frac{ \abs{M(h)} }{(L(h) \psi(h))^{1/2}} >t} \le
    \pi(i_0(\eta), t/2),
  \end{equation*}
  where we put
  \begin{equation*}
    i_0(\eta)=\min\set{i\in\N : h_i < \eta}.
  \end{equation*}
\end{lemma}

\begin{proof}
  Note that $u(h) := (\psi(h) / L(h))^{1/2}$ is decreasing, so $h =
  \Hb_{u(h)}$ for $h \in \HH$, and note that
  \begin{equation*}
    \frac{ \abs{M(h)} }{(L(h) \psi(h))^{1/2}} = u(h)^{-1} \frac{
      \abs{M(\Hb_{u(h)})} }{L(\Hb_{u(h)})}.
  \end{equation*}
  If $h < H^*$ then $u(h) = (\psi(h) / L(h))^{1/2} \geq \bar W(h)$
  using~\eqref{eq:Hstar}, and $\bar W(h) \geq \epsilon_0 (h /
  h_0)^{\alpha_0}$. So, $u(h) \geq \epsilon_0 (H_{u(h)} /
  h_0)^{\alpha_0}$ when $h < H^*$. If $h \geq H_{u_0}$, then $u(h)
  \leq u_0$ using the definition of $H_{u_0}$. This entails
  \begin{align*}
    \sup_{ \Hb_{u_0} \le h < H^*, h \in \HH} &\frac{ \abs{M(h)} }{
      (L(h) \psi(h))^{1/2}} \\
    &\le \sup \Big\{ u^{-1}\frac{ \abs{M(\Hb_u)} }{L(\Hb_u)}; u :
    \Hb_u < H^* \text{ and }\delta_0(H_u / h_0)^{\al_0}<u\le u_0 \Big\}.
  \end{align*}
  Hence, for any $u$ such that $\delta_0(H_u / h_0)^{\al_0} < u \le u_0$
  and $\Hb_u < H^* \le \eta$, one has using~\eqref{eq:Hu}:
  \begin{align*}
    u^{-1} \frac{ \abs{M(\Hb_u)}}{L(\Hb_u)} &\le 2 u^{-1} \frac{
      \abs{M(\Hb_u)}}{L(\Hb_u) + u^{-2} \psi(\Hb_u)} \\
    &=2 \psi(\Hb_u)^{-1/2} Z\bigpar{\Hb_u,u^{-2}\psi(\Hb_u)} \\
    &\le 2 \sup_{i : h_i < \eta} \psi(h_i)^{-1/2} \sup_{
      \delta_0(h_i/h_0)^{\al_0} \le u \le u_0}
    Z\bigpar{h_i,u^{-2}\psi(h_j)}. \qedhere
  \end{align*}
\end{proof}

\begin{lemma} 
  \label{lem:HstarhatH}
  For any $s, t > 0$ define
  \begin{equation}
      \label{eq:etast}
    \eta_{s, t} := h_0 \Big( \frac{u_0 s}{\delta_0 t} \Big)^{1 /
      \alpha_0}.
  \end{equation}
  Then, for all $0 < s < t$, we have:
  \begin{align*}
    \P' &\Big[ H^* > \hat H, \bar W(H^*)^{-1} \bigabs{\hat f(\hat
      H)-f(x)}>t \Big] \\
    &\leq \pi\Big(0,\frac{s - 1}{2} \Big) + \pi \Big( i_0(\eta_{s,
      t}), \frac 14 \Big( \nu - \frac{2s}{t} \Big) \Big) + \pi\Big(0,
    \frac 12 \Big( \frac{\nu t}{2s} - 1 \Big) \Big).
  \end{align*}
\end{lemma}

\begin{proof}
  Let $0<s<t$. One has
  \begin{align*}
    \P' \big[ H^* > \hat H&, \bar W(H^*)^{-1} \bigabs{\hat f(\hat
      H) - f(x)} > t \big] \\
    &\leq \P' \big[ H^* > \hat H, (L(\hat H) / \psi(\hat
    H))^{1/2} \bigabs{\hat f(\hat H)-f(x)}>s \Big] \\
    &+\P'\Big[ H^*>\hat H, (\psi(\hat H) / L(\hat H))^{1/2} > (t/s)
    \bar W(H^*) \Big].
  \end{align*}
  The first term is less than $\pi(0,(s-1)/2)$, indeed, on $\{ W(H^*)\le
  u_0, H^*>\hat H \}$ one has
  \begin{align*}
    ( L(\hat H) / \psi(\hat H))^{1/2} \bigabs{\hat f(\hat H)-f(x)}
    &\le (L(\hat H) / \psi(\hat H))^{1/2} \bigabs{\tilde f(\hat
      H)-f(x)} \\
    & + (L(\hat H) \psi(\hat H))^{-1/2} \abs{ M(\hat H)} \\
    &\le (L(\hat H) / \psi(\hat H))^{1/2} W(\hat H)
    + (L(\hat H) \psi(\hat H))^{-1/2} \abs{ M(\hat H)} \\
    &\le 1 + (L(\hat H) \psi(\hat H))^{-1/2} \abs{ M(\hat H)},
  \end{align*}
  and the desired upper-bound follows from Lemma \ref{majsup}. Let us
  bound the second term. Consider
\begin{equation*}
  \om \in \big\{ W(H^*) \le u_0, H^* > \hat H, (\psi(\hat H) / L(\hat
  H))^{1/2} > (t/s) \bar W(H^*) \big\}.
\end{equation*}
Due to the definition of $\hat H$, see~\eqref{eq:hatH}, there exits
$h' = h'_\om \in [\Hb_{u_0}, \hat H]$ such that
\begin{equation*}
  \bigabs{\hat f(h') - \hat f(\hat H_+)} > \nu (\psi(h') /
  L(h'))^{1/2}.
\end{equation*}
But since $h' \leq \hat H < H^*$, one has
\begin{align*}
  \nu \Big( \frac{\psi(h')}{L(h')} \Big)^{1/2} < \bigabs{\hat f(h') -
    \hat f(\hat H_+)}& \le \bigabs{\tilde f(h')-\tilde f(\hat H_+)}+
  \frac{\abs{M(h')}}{L(h')} +
  \frac{\abs{M(\hat H_+)}}{L(\hat H_+)} \\
  &\le 2 \bar W(H^*) + \frac{\abs{M(h')}}{L(h')} + \frac{\abs{M(\hat
      H_+)}}{L(\hat H_+)} \\
  &\le \frac{2 s}{t} \Big( \frac{\psi(\hat H)}{L(\hat H)} \Big)^{1/2}
  + \frac{\abs{M(h')}}{L(h')} + \frac{\abs{ M(\hat H_+)} }{L(\hat H_+)} \\
  &\le \frac{2 s}{t} \Big( \frac{\psi(h')}{L(h')} \Big)^{1/2} +
  \frac{\abs{M(h')}}{L(h')} + \frac{\abs{M(\hat H_+)}}{L(\hat H_+)}.
\end{align*}
So, since $h' \leq \hat H$ entails (for such an $\omega$) that
$(\psi(h') / L(h'))^{1/2} \geq (\psi(\hat H) / L(\hat H))^{1/2} > (t /
s) \bar W(H^*)$, we obtain
\begin{align*}
  \frac{\abs{M(h')}}{L(h')} + \frac{\abs{M(\hat H_+)}}{L(\hat H_+)} &>
  \Big(\nu - \frac{2s}{t} \Big) \Big(\frac{\psi(h')}{L(h')}
  \Big)^{1/2} \\
  &\ge \Big(\nu - \frac{2s}{t} \Big) \max \Big[
  \Big(\frac{\psi(h')}{L(h')} \Big)^{1/2}, \frac ts \bar W(H^*) \Big],
\end{align*}
and therefore
\begin{align*}
  \om \in &\Big\{ \sup_{ \Hb_{u_0} \le h < H^*, h \in \HH} \frac{
    \abs{M(h)} }{(L(h) \psi(h))^{1/2}} > \frac 12 \Big(\nu -
  \frac{2s}{t} \Big) \Big\} \\
  &\cup \Big\{ \frac{\abs{M(H^*)}}{L(H^*)} \ge \frac{t}{2s} \Big(\nu -
  \frac{2s}{t} \Big) \bar W(H^*) \Big\}.
\end{align*}
In addition, because of $\hat H \ge \Hb_{u_0}$ one has
\begin{equation*}
  \delta_0(H^* / h_0)^{\al_0} \le \bar W(H^*) < (s / t) (\psi(\hat H) /
  L(\hat H)) ^{1/2} \le (s/t) u_0,
\end{equation*}
so $H^* \le \eta_{s, t}$, where $\eta_{s, t}$ is given
by~\eqref{eq:etast}. We have shown that
\begin{align*}
  &\Big\{ W(H^*) \le u_0, H^* > \hat H, \Big( \frac{\psi(\hat
    H)}{L(\hat H)} \Big)^{1/2}
  > \frac ts \bar W(H^*) \Big\}  \\
  &\subset \Big\{ H^* \le \eta_{s, t}, \sup_{ \Hb_{u_0} \le h < H^*,
    h\in\HH} \frac{ \abs{M(h)} }{(L(h) \psi(h))^{1/2}} > \frac 12
  \Big(\nu -
  \frac{2s}{t} \Big) \Big\} \\
  &\cup \Big\{ \frac{\abs{M( H^*)}}{L( H^*)} \ge \Big( \frac{\nu
    t}{2s} - 1 \Big) \bar W(H^*) \Big\},
\end{align*}
and we conclude using Lemma~\ref{majsup} and \eqref{Mh}.
\end{proof}

\subsubsection{Finalization of the proof}
\label{sec:finalization_proof_thm1}

In order to conclude the proof of Theorem~\ref{thm:upper_bound}, we
need the following uniform version of Theorem~\ref{thm:key}: under the
same assumptions as in Theorem~\ref{thm:key}, we have for any $0 < a_0
< a_1$:
\begin{equation}
  \label{eq:psi2uniform}
  \E \Bigcro{\sup_{a \in [a_0, a_1]} \exp\Bigpar{\frac{\la}{2} \frac{a
        M_N^2}{(a + V_N)^2} }} \le (1 + c_{\lambda})(1 + \log(a_1 /
  a_0)).
\end{equation}
Indeed, since
\begin{equation*}
  \Bigabs{ \frac{\partial}{\partial a} \frac{a M_N^2}{(a + V_N)^2} } =
  \Bigabs{ \frac{M_N^2}{(a + V_N)^3} (V_N - a)}
  \le a^{-1} \frac{a M_N^2}{(a + V_N)^2} = Y^a / a,
\end{equation*}
we have
\begin{align*} 
  \sup_{a\in[a_0,a_1]} \exp( \la Y^{a}/2) &\le \exp( \la Y^{a_0}/2) +
  \int_{a_0}^{a_1} a^{-1} \exp( \la Y^a/2) \la Y^a / 2  \,da\\
  &\le \exp( \la Y^{a_0}) + \int_{a_0}^{a_1} a^{-1} \exp( \la Y^a)\,
  da,
\end{align*}
so~\eqref{eq:psi2uniform} follows taking the expectation and using
Theorem~\ref{thm:key}. Now, using~\eqref{eq:psi2uniform} with
\begin{equation*}
  s_k = \frac{1}{\si_{k-1}} \ind{\abs{X_{k-1}-x}\le h}, \quad
  \ze_k = \ep_k/{\si_{k-1}}
\end{equation*}
we obtain
\begin{equation*}
  \E \Big[ \exp\bigpar{(\la/2)  \sup_{a \in [a_0, a_1]} Z(h,a)^2}
  \Big] \le C(1+\log(a_1/a_0)),
\end{equation*}
where we recall that $Z(h, a)$ is given by~\eqref{eq:MZdef}. So, using
Markov's inequality, we arrive, for all $h>0$, $a_1 > a_0 > 0$ and
$t\ge 0$, at:
\begin{equation}
  \label{eq:uniform_markov}
  \P \Big[ \sup_{a \in [a_0, a_1]} Z(h,a) \ge t \Big] \le C
  (1+\log(a_1/a_0)) e^{-\la t^2/2}.
\end{equation}
A consequence of \eqref{eq:uniform_markov}, together with an union
bound, is that for all $i_0 \in \N$ and $t>0$:
\begin{equation}
  \label{eq:majZ}
  \pi(i_0,t) \le C e^{-\la t^2/2} \sum_{i\ge i_0} (h_i/h_0)^{b\la t^2/2}
  \bigpar{1 + 2\log(u_0/\delta_0) + 2\al_0 \log(h_0/h_i)},
\end{equation}
where we recall that $\pi(i_0, t)$ is given by~\eqref{eq:pidef}.

Now, it remains to use what the grid $\mathcal H$ is. Recall that for
some $q \in (0,1)$, we have $h_i=h_0 q^i$ and we denote by $C$ any
positive number which depends only on $\la, \mu,\ga, q, b, u_0,
\delta_0, \al_0, \nu$. Using together~\eqref{eq:HstarhatH} and
Lemma~\ref{lem:HstarhatH}, one gets for $0 < s < t$:
\begin{align*}
  \P' \big[ \bar W(H^*)^{-1} \abs{ \hat f (\hat H)- f(x) } >t \big]
  &\leq \pi \Big( 0, \frac{t - \nu - 1}{2} \Big) + \pi\Big(0,\frac{s -
    1}{2} \Big) \\
  &+ \pi \Big( i_0(\eta_{s, t}), \frac 14 \Big( \nu - \frac{2s}{t}
  \Big) \Big) + \pi\Big(0, \frac 12 \Big( \frac{\nu t}{2s} - 1 \Big)
  \Big),
\end{align*}
and using~\eqref{eq:majZ}, we have for any $u>0$, $i_0 \in \mathbb N$:
\begin{equation*}
  \pi(i_0, u) \le C e^{-\la u^2/2} (i_0 + 1) q^{i_0 b \la u^2/2}.
\end{equation*}
Recalling that $\eta_{s, t}$ is given by~\eqref{eq:etast} and that
$i_0(\eta)=\min\set{i\in\N : h_i < \eta}$, we have
\begin{equation}
  \label{eq:i0upperlower}
  \frac{\log(\delta_0 / u_0) + \log(t / s)}{\alpha_0 \log(1 / q)} <
  i_0(\eta_{s, t}) \leq   \frac{\log(\delta_0 / u_0) + \log(t /
    s)}{\alpha_0 \log(1 / q)} + 1.
\end{equation}
Now, recall that $0 < \rho < \frac{b \mu \nu^2}{64 \alpha_0 (1 +
  \gamma)}$ and consider $s = \sqrt{(8 \rho \log t) / \lambda} +1$.
When $t$ is large enough, we have $s < t$ and:
\begin{align*}
  &\pi\Big(0, \frac{s - 1}{2} \Big) \leq C_1 t^{-\rho}, \quad \pi
  \Big( 0, \frac{t - \nu - 1}{2} \Big) \leq C_2 \exp( - C_2' t^2), \\
  &\pi\Big(0, \frac 12 \Big( \frac{\nu t}{2s} - 1 \Big) \Big) \leq C_3
  \exp\bigpar{ -C_3' (t / \log t)^2},
\end{align*}
for constants $C_i, C_i'$ that depends on $\lambda, b, \nu, \delta_0,
u_0, \alpha_0, q$. For the last probability, we have:
\begin{align*}
  \pi \Big( i_0(\eta_{s, t}), \frac 14 \Big( \nu - \frac{2s}{t} \Big)
  \Big) &\leq C \exp\Big( - \frac{\lambda (\nu - 2s / t)^2}{32} \Big)
  (i_0(\eta_{s, t}) + 1) \\
  & \quad \quad \times \exp\Big( -\frac{i_0(\eta_{s, t}) b \lambda
    (\nu - 2s / t)^2 \log(1/q)}{32} \Big),
\end{align*}
and by taking $\lambda \in (0, \frac{\mu}{2(1 + \gamma)})$ and $t$
large enough, one has
\begin{equation*}
  \frac{b \lambda (\nu - 2 s / t)^2}{32 \alpha_0} > \rho,
\end{equation*}
so we obtain together with~\eqref{eq:i0upperlower}:
\begin{align*}
  \pi \Big( i_0(\eta_{s, t}), \frac 14 \Big( \nu - \frac{2s}{t} \Big)
  \Big) &\leq C \frac{(\log(t+1))^{1 + \rho/2}}{t^{\rho}},
\end{align*}
when $t$ is large enough. This concludes the proof of
Theorem~\ref{thm:upper_bound}. $\hfill \square$

\subsection{Proof of Proposition~\ref{prop:det-rate}}

Let us denote for short $I_h = [x - h, x + h]$. Recall that $h_w$ is
well-defined when $n \geq \sigma^2 / (P_X[I_{h_0}] w(h_0)^{2})$, and
that $H_w$ is well defined on the event
\begin{equation*}
  \Omega_0 = \{ L(h_0) \geq w(h_0)^{-2} \}.
\end{equation*}
So, from now on, we suppose that $n$ is large enough, and we work on
$\Omega_0$. We need the following Lemma, which says that, when
$L(h_w)$ and $\E L (h_w)$ are close, then $H_w$ and $h_w$ are close.
\begin{lemma}
  \label{lem:bandwdith-embedding}
  If Assumption~\ref{ass:fsmoothness} holds, we have for any $0 <
  \epsilon < 1$ that on $\Omega_0$:
  \begin{align*}
    &\Big\{ L(h_w) \geq \frac{\E L(h_w)}{(1 + \epsilon)^s} \Big\}
    \subset \big\{ H_w \leq (1 + \varepsilon) h_w \big\} \quad
    \text{ and } \\
    & \Big\{ L(h_w) \leq \frac{\E L(h_w)}{(1 - \epsilon)^s} \Big\}
    \subset \big\{ H_w > (1 - \varepsilon) h_w \big\},
  \end{align*}
  when $n$ is large enough.
\end{lemma}
The proof of Lemma~\ref{lem:bandwdith-embedding} is given in
Section~\ref{sec:lemmas} below.  We use also the next Lemma from
\cite{MR1865343} (see Claim~2, p.~858). It is a corollary of Berbee's
coupling lemma \cite{MR547109}, that uses a construction from the
proof of Proposition~5.1 in \cite{MR1440142}, see p.~484.
  \begin{lemma}
    \label{lem:coupling}
    Grant Assumption~\ref{ass:Xbeta-mixing}. Let $q, q_1$ be integers
    such that $0 \leq q_1 \leq q / 2$, $q_1 \geq 1$. Then, there exist
    random variables $(X^*_i)_{i=1}^n$ satisfying the following:
    \begin{itemize}
    \item For $j = 1, \ldots, J := [n / q],$ the random vectors
      \begin{equation*}
        U_{j, 1} := (X_{(j-1) q + 1}, \ldots, X_{(j-1)q + q_1}) \; \text{
          and } \; U_{j, 1}^* := (X_{(j-1)q + 1}^*, \ldots, X_{(j-1)q +
          q_1}^*)
      \end{equation*}
      have the same distribution, and so have the random vectors
      \begin{equation*}
        U_{j, 2} := (X_{(j-1) q + q_1 + 1}, \ldots, X_{j q}) \; \text{
          and } \; U_{j, 2}^* := (X_{(j-1)q + q_1 + 1}^*, \ldots, X_{j q}^*).
      \end{equation*}
    \item For $j = 1, \ldots, J$,
      \begin{equation*}
        \P[ U_{j, 1} \neq U_{j, 1}^*] \leq \beta_{q - q_1} \; \text{
          and } \; \P[ U_{j, 2} \neq U_{j, 2}^*] \leq \beta_{q_1}.
      \end{equation*}
    \item For each $k = 1, 2$, the random vectors $ U_{1, k}^*,
      \ldots, U_{J, k}^*$ are independent.
    \end{itemize}
  \end{lemma}
  In what follows, we take simply $q_1 = [q / 2] + 1$, where $[x]$
  stands for the integral part of $x$, and introduce the event
  $\Omega^* = \{ X_i = X_i^*, \forall i = 1, \ldots, n \}$. Assume to
  simplify that $n = J q$. Lemma~\ref{lem:coupling} gives
  \begin{equation}
    \label{eq:controlPstar}
    \P [ (\Omega^*)^\complement] \leq J ( \beta_{q - q_1} + \beta_{q -
      q_1}) \leq 2 J \beta_{[q / 2]} \leq \frac{2 n \beta_{[q /
        2]}}{q}.
  \end{equation}
  Then, denote for short $L^*(h) = \sum_{i=1}^n \ind{|X_{i-1}^* - x|
    \leq h}$, and note that, using
  Lemma~\ref{lem:bandwdith-embedding}, we have, for $z := 1 - 1 / (1 +
  \epsilon)^s$:
  \begin{align*}
    \big\{ H_w > (1 + \epsilon)^s h_w \big\} \cap \Omega^* \cap
    \Omega_0 &\subset
    \big\{ L^*(h_w) - \E L(h_w) \geq z \E L(h_w) \big\} \\
    &= \Big\{ \frac{1}{n} \sum_{i=1}^n ( \ind{|X_{i-1}^* - x| \leq
      h_w} - P_X[I_{h_w}] ) \geq z P_X[I_{h_w}] \Big\}.
  \end{align*}
  Use the following decomposition of the sum:
  \begin{align*}
    \frac{1}{n} \sum_{i=1}^n ( \ind{|X_{i-1}^* - x| \leq h_w} -
    P_X[I_{h_w}] ) \leq \frac{1}{J} \sum_{j=1}^J ( Z_{j, 1} + Z_{j,2}
    ),
  \end{align*}
  where for $k \in \{ 1, 2 \}$, we put
  \begin{align*}
    Z_{j, k} := \frac 1q \sum_{i \in I_{j, k}} ( \ind{|X_{i-1}^* - x|
      \leq h_w} - P_X[I_{h_w}] ),
  \end{align*}
  where $I_{j, 1} := \{ (j-1) q + 1, \ldots, (j-1) q + q_1 \}$ and $
  I_{j, 2} := \{ (j-1) q + q_1 + 1, \ldots, j q \}.$ For $k \in \{ 1,
  2 \}$, we have using Lemma~\ref{lem:coupling} that the variables
  $(Z_{j, k})_{j=1}^J$ are independent, centered, such that
  $\norm{Z_{j, k}}_\infty \leq 1/2$ and $\E[Z_{j, k}^2] \leq
  P_X[I_{h_w}] / 4$. So, Bernstein's inequality gives
  \begin{equation*}
    \P \Big[ \Big\{ H_w > \frac{1}{(1 - z)^{1 / s}} h_w \Big\} \cap
    \Omega^* \cap \Omega_0 \Big] \leq 2 \exp \Big( - \frac{z^2}{2(1 +
      z / 3)} \frac{n P_X[I_{h_w}]}{q} \Big),
  \end{equation*}
  and doing the same on the other side gives for any $z \in (0, 1)$:
  \begin{equation*}
    \P \Big[ \Big\{ \frac{1}{(1 + z)^{1 / s}} h_w \leq H_w \leq
    \frac{1}{(1 - z)^{1 / s}} h_w \Big\}^\complement \cap \Omega^*
    \big] \leq 4 \exp \Big( - \frac{z^2}{2(1 + z / 3)} \frac{n
      P_X[I_{h_w}]}{q} \Big).
  \end{equation*}
  So, when $n$ is large enough, we have
  \begin{equation}
    \label{eq:Hwhw}
    \P [ h_w / 2 \leq H_w \leq 2 h_w ] \geq 1 - 4 \exp \Big(- \frac{C n
      P_X[I_{h_w}]}{q} \Big) - \frac{2 n \beta_{[q / 2]}}{q}.
  \end{equation}
  But, since on $[0, h_0]$ $w$ is increasing and $w(h) = h^s
  \ell_w(h)$ where $\ell_w$ is slowly varying, we have $\{ h_w / 2
  \leq H_w \leq 2 h_w \} \subset \{ w(h_w) / 4 \leq w(H_w) \leq 4
  w(h_w) \}$ when $n$ is large enough. Now, Lemma~\ref{lem:hw} and
  Assumption~\ref{ass:behavior-PX} gives that
  \begin{equation*}
    n P_X[I_{h_w}] =  n^{2s / (2s + \tau + 1)} \ell(1 / n),
  \end{equation*}
  where $\ell$ is a slowly varying function that depends on $\ell_X$,
  $\ell_w$, $s$, $\tau$ and $\sigma$. When the $\beta$-mixing is
  geometric, we have $\psi^{-1}(p) = \exp((p / \eta)^{1/\kappa})$, so
  the choice $q = n^{2 s \kappa / ((2s + \tau + 1)(\kappa+1))}$
  implies
  \begin{equation*}
    \P \Big[ \{ \frac{w(h_w)}{4} \leq w(H_w) \leq 4 w(h_w) \} \cap
    \Omega_0 \Big] \geq 1 - \exp( -C_1 n^{\delta_1} \ell_1(1 / n)).
  \end{equation*}
  When the mixing is arithmetic, we have $\psi^{-1}(p) = (p /
  \eta)^{1/\kappa}$, so the choice $q = n^{2s / (2s + \tau + 1)}
  \ell(1 / n) / (\log n)^2$ implies
  \begin{equation*}
    \P \Big[ \{ \frac{w(h_w)}{4} \leq w(H_w) \leq 4 w(h_w) \} \cap
    \Omega_0 \Big] \geq 1 -C_2 n^{-\delta_2} \ell_2(1/n).
  \end{equation*}
  So, it only remains to control the probability of $\Omega_0$. Using
  the same coupling argument as before together with Bernstein's
  inequality, we have when $n$ is large enough:
  \begin{align*}
    \P[ L(h_0) < w(h_0)^{-2} ] &= \P[ L(h_0) - \E L(h_0) < w(h_0)^{-2}
    - \E L(h_0)] \\
    &\leq \P\Big( L(h_0) - \E L(h_0) < -\frac{n P_X[I_{h_0}]}{2} \Big)
    \\
    &\leq \exp\Big( - C_2 \frac{n P_X[I_{h_0}]}{q}\Big) + \frac{2 n
      \beta_{[q/2]}}{q}.
  \end{align*}
  So, when the $\beta$-mixing is geometric, the choice $q = n^{\kappa
    / (\kappa+1)}$ implies that $\P[ \Omega_0^\complement ] \leq \exp(
  - C_1 n^{1 / (\kappa+1)}) = o(\phi_n)$. When the mixing is
  arithmetic, we have $\psi^{-1}(p) = (p / \eta)^{1/\kappa}$, so the
  choice $q = n / (\log n)^2$ gives $\P[ \Omega_0^\complement ] \leq
  C_2 (\log n)^2 n^{-1 / \kappa} = o(\phi_n)$. This concludes the
  proof of Proposition~\ref{prop:det-rate}. $\hfill \square$

\subsection{Proof of Corollary~\ref{cor:det-upper-bound}}

Let us fix $\rho \in (p, \frac{b \mu \nu^2}{128 (1 + \gamma)})$ (note
that $\alpha_0 = 2$ under Assumption~\ref{ass:fsmoothness}). Using
Assumption~\ref{ass:fsmoothness}, one can replace $\bar W$ by $w$ in
the statement of Theorem~\ref{thm:upper_bound}. This gives
\begin{equation*}
  \P \Big[ \Big\{ \abs{ \hat f (\hat H) - f(x) } \geq t w(H^*)
  \Big\} \cap \Omega_0 \Big] \leq C_0 \frac{(\log(t+1))^{\rho/2 +
      1}}{t^\rho}
\end{equation*}
for any $t \geq t_0$, where we recall that $\Omega_0 = \{
L(h_0)^{-1/2} \leq w(h_0) \}$, and where
\begin{equation*}
  % \label{eq:Hstar}
  H^* := \min \Big\{ h \in \HH : \Big( \frac{\psi(h)}{L(h)} \Big)^{1/2}
    \le w(h) \Big\}.
\end{equation*}
Recall the definition \eqref{eq:Hw} of $H_w$, and note that by
construction of $\mathcal H$, one has that $H_w \leq H^* \leq q^{-1}
H_w$. So, on the event $\{ H_w \leq 2 h_w \}$, one has, using the fact
that $w$ is $s$-regularly varying, that $w(H^*) \leq w(2 q^{-1} h_w)
\leq 2 (2 / q)^s w(h_w)$ for $n$ large enough. So, putting for short
$A := \{ H_w \leq 2 h_w \} \cap \Omega_0$, we have
\begin{equation*}
  \P \Big[ \Big\{ \abs{ \hat f (\hat H) - f(x) } \geq c_1 t w(h_w)
  \Big\} \cap A \Big] \leq C_0 \frac{(\log(t+1))^{\rho/2 +
      1}}{t^\rho}
\end{equation*}
for any $t \geq t_0$, where $c_1 = 2 (2 / q)^s$. Since $\rho > p$, we
obtain, by integrating with respect to $t$, that
\begin{equation*}
  \E \big[ | w(h_w)^{-1} (\hat f (\hat H) - f(x)) |^p \ind{A} \big]
  \leq C_1,
\end{equation*}
where $C_1$ is a constant depending on $C_0, t_0, q, \rho, s, p$. Now,
it only remains to observe that using Proposition~\ref{prop:det-rate},
$\P(A^\complement) \leq 2 \varphi_n$, and that $\varphi_n = o(w(h_w))$
in the geometrically $\beta$-mixing case, and in the arithmetically
$\beta$-mixing when $\kappa < 2s / (s + \tau + 1)$. $\hfill \square$

\section{Proof of the Lemmas}
\label{sec:lemmas}

\subsection{Proof of Lemma~\ref{lem:bandwdith-embedding}}

  For $n$ large enough, we have $\psi((1 + \epsilon) h_w) / \ell_w( (1
  + \epsilon) h_w)^2 \leq (1 + \epsilon)^s \psi(h_w) / \ell_w(h_w)^2$
  since $\psi / \ell_w^2$ is slowly varying. So,
\begin{equation*}
  \frac{\psi((1 + \epsilon) h_w)}{w( (1 + \epsilon) h_w)^2} \leq
  \frac{1}{(1 + \epsilon)^s} \frac{\psi(h_w)}{w(h_w)^2} =
  \frac{1}{(1 + \epsilon)^s} \E L(h_w).
\end{equation*}
On the other hand, by definition of $H_w$, we have
\begin{equation*}
  \{ H_w \leq (1 + \epsilon) h_w \} = \Big\{ L( (1 + \epsilon) h_w) \geq
  \frac{\psi( (1 + \epsilon) h_w ) }{w( (1 + \epsilon) h_w)^2} \Big\},
\end{equation*}
and $L( (1 + \epsilon) h_w) \geq L(h_w)$, so we proved that the
embedding
\begin{equation*}
  \Big \{\frac{L(h_w)}{\E L(h_w)} \geq \frac{1}{(1 + \epsilon)^s}
  \Big\} \subset \{ H_w \leq (1 + \epsilon) h_w \}
\end{equation*}
holds when $n$ is large enough. The same argument allows to prove that
\begin{equation*}
  \Big \{\frac{L(h_w)}{\E L(h_w)} \leq \frac{1}{(1 - \epsilon)^s}
  \Big\} \subset
  \{ H_w > (1 - \epsilon) h_w \},
\end{equation*}
which concludes the proof of the Lemma.  $\hfill \square$

\subsection{Proof of Lemma \ref{l1}}

  Take $m \in[0, \mu)$ and $\rho \in \R$. Note that $e^y \le 1 +
  ye^y \le 1 + y + y^2 e^y$ for any $y \geq 0$, so
  \begin{align*}
    e^{m \ze^2 + \rho\ze} &\leq e^{\rho \ze} + m \ze^2 e^{m \ze^2 +
      \rho \ze} \\
    &\le 1 + \rho\ze + (\rho^2+m) \ze^2 e^{m \ze^2 + \rho\ze},
  \end{align*}
  and
  \begin{equation}
    \label{majbase}
    \E[ e^{m \ze^2 + \rho\ze}] \le 1+ (\rho^2+m) \E[\ze^2 e^{m \ze^2 +
      \rho\ze}],
  \end{equation}
  since $\E \ze = 0$. Take $m_1 \in (m, \mu)$. Since $\rho \zeta \leq
  \epsilon \rho^2 / 2 + \zeta^2 / (2 \epsilon)$ for any $\epsilon >
  0$, we obtain for $\epsilon = [2(m_1 - m)]^{-1}$:
  \begin{equation*}
    e^{m\ze^2+\rho\ze} \le \exp\bigpar{\frac{\rho^2}{4(m_1-m)}}
    e^{m_1\ze^2}.
  \end{equation*}
  Together with
  \begin{equation*}
    \ze^2 \le \frac{1}{\mu - m_1} e^{(\mu - m_1)\ze^2}
  \end{equation*}
  and the definition of $\mu$, this entails
  \begin{equation*}
    \E[ \ze^2 e^{m \ze^2 + \rho\ze}] \le \frac{\gamma}{\mu - m_1}
    \exp\bigpar{\frac{\rho^2}{4(m_1-m)}}.
  \end{equation*}
  Thus, 
  \begin{align*}
    \E[e^{m \ze^2 + \rho\ze}] & \le 1 + \frac{\gamma(\rho^2 + m)
    }{\mu-m_1} \exp\bigpar{\frac{\rho^2}{4(m_1-m)}} \\
    &\le 1 + \frac{\gamma(\rho^2+m) }{\mu-m_1}
    \exp\bigpar{\frac{\rho^2+m}{4(m_1-m)}}.
  \end{align*}
  For the choice $m_1 = \mu / (1 + 2 \gamma) + 2 \gamma m / (1 + 2
  \gamma)$ one has $\gamma / (\mu - m_1) = 1 / [2(m_1 - m)]$, so the
  Lemma follows using that $1 + ye^{y/2} \le e^{y}$ for all $y\ge 0$.
  This concludes the proof of the Lemma.  $\hfill \square$

\subsection{Proof of Lemma \ref{calcul}}

  Let $\eta\in[0,1]$ and $z\in\R_+$ be such that
  $e^{A\eta}\cosh((1-\eta)z)-\cosh(z) \ge 0$.  Let us show that one
  has
  \begin{equation}
    \label{majz}
    z \le 2 \log 2 + 2 A.
  \end{equation}
  Since $\cosh(z)/\cosh((1-\eta)z) \ge e^{\eta z}/2$ one has $z\le
  \eta^{-1} \log 2 + A$. Thus \eqref{majz} holds if $\eta\ge 1/2$. If
  $\eta<1/2$ and $z \geq \log(3)$, it is easy to check that the
  derivative of $x \mapsto \cosh((1 - x)z) e^{\eta x/2}$ is
  non-positive, hence $\cosh(z)\ge e^{\eta z/2} \cosh((1-\eta)z)$ in
  this case. Thus, we have either $z \le \log(3)$ or $z\le 2A$ which
  yields \eqref{majz} in every case. Finally, from \eqref{majz}, we
  easily derive
  \begin{align*}
    e^{A\eta} \cosh((1-\eta)z) - \cosh(z) &= \cosh ((1-\eta)z)
    \Big( e^{A\eta} - \frac{\cosh(z)}{\cosh((1-\eta)z)} \Big) \\
    &\le \cosh(z) (e^{A\eta} - 1) \\
    &\le \cosh\bigpar{2\log(2) + 2 A} A\eta e^{A\eta}. \qedhere
  \end{align*}
  This concludes the proof of the Lemma.  $\hfill \square$

\bibliography{biblio}

\end{document}